\documentclass[10pt]{amsart}

\usepackage{enumerate, amsmath, amsfonts, amssymb, amsthm, mathtools, thmtools, wasysym, graphics, graphicx, xcolor, frcursive,xparse,comment,bbm}
\usepackage{tabmac}
\usepackage[all]{xy}
\usepackage[backref=page]{hyperref}

\usepackage{url, hypcap}
\hypersetup{colorlinks=true, citecolor=darkblue, linkcolor=darkblue}

\usepackage{tikz}
\usetikzlibrary{calc,through,backgrounds,shapes,matrix}
\usepackage{graphicx}
\graphicspath{{Fig/}}
\DeclareGraphicsExtensions{.pdf}

\definecolor{darkblue}{rgb}{0.0,0,0.7} 
\newcommand{\darkblue}{\color{darkblue}} 
\definecolor{lightblue}{rgb}{0,135,147}
\newcommand{\lightblue}{\color{lightblue}}
\definecolor{red}{rgb}{1,0,0}

\definecolor{darkred}{rgb}{0.7,0,0} 
\newcommand{\darkred}{\color{darkred}} 
\definecolor{lightgrey}{rgb}{0.7,0.7,0.7} 

\newcommand{\stepl}{ -- ++(-1,0)}

\newcommand{\stepu}{ -- ++(0,1)}
\newcommand{\stepd}{ -- ++(0,-1)}

\newdimen\squaresize \squaresize=12pt
\def\lb{\lightblue{\vrule height\squaresize width\squaresize}}

\newtheorem{theorem}{Theorem}[section]
\newtheorem{proposition}[theorem]{Proposition}
\newtheorem{corollary}[theorem]{Corollary}
\newtheorem{lemma}[theorem]{Lemma}

\theoremstyle{definition}
\newtheorem{definition}[theorem]{Definition}
\newtheorem{example}[theorem]{Example}

\newtheorem{openproblem}[theorem]{Open Problem}
\newtheorem{remark}[theorem]{Remark}
\numberwithin{equation}{section}

\usepackage[T1]{fontenc}

\renewcommand{\L}{\mathcal{L}}
\newcommand{\R}{\mathcal{R}}

\newcommand{\defn}[1]{\emph{\darkred #1}} 
\newcommand{\des}{{\sf des}}
\newcommand{\Red}{{{\sf Red}}}
\newcommand{\ShSYT}{{{\sf ShSYT}}}
\newcommand{\rSYT}{{{\sf rSYT}}}
\newcommand{\hrSYT}{{{\sf hrSYT}}}

\newcommand{\rW}{{{\sf rW}}}

\newcommand{\w}{\mathbf{w}}

\usepackage[colorinlistoftodos]{todonotes}

\title[Braid Moves in Commutation Classes]
  {Braid Moves in Commutation Classes of the Symmetric Group}

\author[A.~Schilling]{Anne Schilling}
\address[A. Schilling]{Department of Mathematics, UC Davis, One Shields Ave., Davis, CA 95616-8633, U.S.A.}
\email{anne@math.ucdavis.edu}

\author[N.~Thi\'ery]{Nicolas M. Thi\'ery}
\address[N. Thi\'ery]{Univ Paris-Sud, Laboratoire de Recherche en Informatique,
  Orsay, F-91405; CNRS, Orsay, F-91405, France}
\email{Nicolas.Thiery@u-psud.fr}

\author[G.~White]{Graham White}
\address[G. White]{Department of Mathematics, Stanford University, 450 Serra Mall, Bldg. 380, 
  Stanford, CA 94305-2125, U.S.A.}
\email{grwhite@math.stanford.edu}

\author[N.~Williams]{Nathan Williams}
\address[N. Williams]{LaCIM, Universit\'e de Qu\'ebec \`a Montr\'eal, Montr\'eal (Qu\'ebec), Canada; New address: Department of Mathematics, UC Santa Barbara, Santa Barbara, CA 93106, U.S.A.}
\email{nathan.f.williams@gmail.com}

\date{\today}
\keywords{}
\subjclass[2000]{Primary 05E45; Secondary 20F55, 13F60}

\begin{document}

\begin{abstract}
We prove that the expected number of braid moves in the commutation class of the reduced word 
$(s_1 s_2 \cdots s_{n-1})(s_1 s_2 \cdots s_{n-2}) \cdots (s_1 s_2)(s_1)$ for the long element in the 
symmetric group $\mathfrak{S}_n$ is one. This is a variant of a similar result by V.~Reiner, who proved 
that the expected number of braid moves in a random reduced word for the long element is one. The proof
is bijective and uses X.~Viennot's theory of heaps and variants of the promotion operator. In addition, we provide
a refinement of this result on orbits under the action of even and odd promotion operators. This gives an example 
of a homomesy for a nonabelian (dihedral) group that is not induced by an abelian subgroup. Our techniques 
extend to more general posets and to other statistics.
\end{abstract}

\maketitle

\section{Introduction}

\subsection{Reduced Words and Standard Tableaux}
Fix the \defn{symmetric group} $\mathfrak{S}_n$ and its generating set of \defn{simple transpositions} 
$S := \{s_i \mid 1\le i <n \}$. The simple transpositions $s_i := (i,i+1)$ satisfy the quadratic relations $s_i^2 = 1$, 
the \defn{commutations} $s_i s_j = s_j s_i$ for $|i-j|>1$, and the \defn{braid moves}
\[
	s_i s_{i+1} s_i = s_{i+1} s_i s_{i+1} \quad \text{for $1\le i\le n-2$.}
\]
The \defn{length} $\ell(w)$ of an element $w \in \mathfrak{S}_n$ is the smallest nonnegative integer $\ell$ for 
which there exists an expression $w = s_{i_1}s_{i_2}\cdots s_{i_\ell}$.  The symmetric group $\mathfrak{S}_n$ 
has a \defn{longest element} $w_0$, whose length is $\ell(w_0) = N := \frac{n(n-1)}{2}.$

If $w \in \mathfrak{S}_n$ can be written as a product of generators $w=s_{i_1}s_{i_2}\cdots s_{i_\ell}$, then 
$\mathbf{w} = i_1 \; i_2 \; \ldots \; i_\ell$ is a \defn{word} for $w$. If the length $\ell$ of the word $\mathbf{w}$ is 
equal to $\ell(w)$, then $\mathbf{w}$ is a \defn{reduced word}.
We may refer to the product of generators $s_{i_1}s_{i_2}\cdots s_{i_{\ell(w)}}$ as a reduced expression or reduced 
word for $w$. By Matsumoto's theorem, the set of reduced words for $w$ form a connected graph $\Red(w)$ with 
edges given by commutation and braid moves. Figure~\ref{fig:matsumoto_graph_S4} illustrates the graph $\Red(w_0)$ 
for $\mathfrak{S}_4$. 

\begin{figure}[t]
	\begin{center}
		\includegraphics[height=3in]{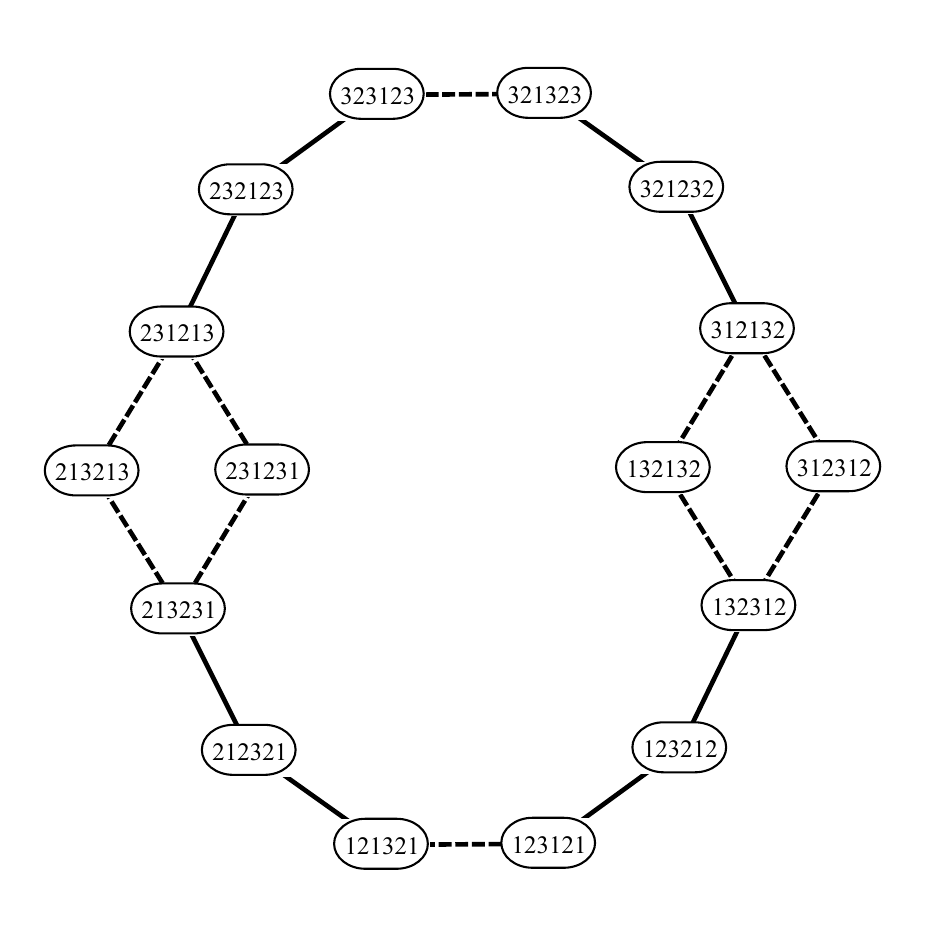}
	\end{center}
\caption{The 16 reduced words in $\Red(w_0)$ for $\mathfrak{S}_4$.  Solid lines denote braid relations, 
while dotted lines indicate commutation relations.}
\label{fig:matsumoto_graph_S4}
\end{figure}

It is natural to ask how many edges in this graph correspond to braid moves.  In~\cite{reiner.2005}, V.~Reiner 
proved the following striking theorem, relating the number of such edges to the number of vertices.

\begin{theorem}[V. Reiner~\cite{reiner.2005}]
	The expected number of braid moves for a reduced word for $w_0 \in \mathfrak{S}_n$ is one.
\label{thm:braid_moves_reiner}
\end{theorem}

In other words, there are $\frac{1}{2}|\Red(w_0)|$ edges that correspond to braid moves in the graph $\Red(w_0)$.  

V.~Reiner's proof relies on P.~Edelman and C.~Greene's equivariant bijection~\cite{edelmann.greene.1987} between 
reduced words for $w_0$ under the action
\[
	s_{i_N}\cdots s_{i_2} s_{i_1} \mapsto s_{n-i_1} s_{i_N}\cdots s_{i_2}
\]
and standard Young tableaux (SYT) of staircase shape $(n-1,n-2,\ldots,1)$ under promotion.  Briefly, he rotates 
the desired braid move to the beginning of the reduced word, so that under the bijection to SYT the braid move 
is sent to a \defn{standard braid hook}---three cells arranged in the shape $(2,1)$, touching the diagonal, and
labeled by consecutive numbers $i-1,i,i+1$.  By excising these three cells, it is possible to compute the 
desired quantity as an explicit summation of a quotient of hook-length formulas.

\subsection{Commutation Classes and Right-Justified Tableaux}
Given a reduced word $\mathbf{w}$ for $w\in \mathfrak{S}_n$, we can form the subgraph $\Red(\mathbf{w})$ 
of $\Red(w)$ containing $\mathbf{w}$ and all reduced words connected to $\mathbf{w}$ using only commutations;
 this is called the \defn{commutation class} of $\mathbf{w}.$  We may now ask for the number of edges emanating 
from this subgraph (which, by construction, necessarily correspond to braid moves).  
Figure~\ref{fig:matsumoto_graph_commutation_class_S5} illustrates an example of such a subgraph for 
$\mathfrak{S}_5$.

In general, it is unreasonable to expect as tidy an answer as the one given in Theorem~\ref{thm:braid_moves_reiner}.  
In fact, the expected number of braid moves is not equal to one on arbitrary commutation
classes (this is already evident in Figure~\ref{fig:matsumoto_graph_S4}). However, there is a special
commutation class where this is true, as stated in the following attractive specialization of our main result.

\begin{figure}[htbp]
	\begin{center}
		\includegraphics[width=\textwidth]{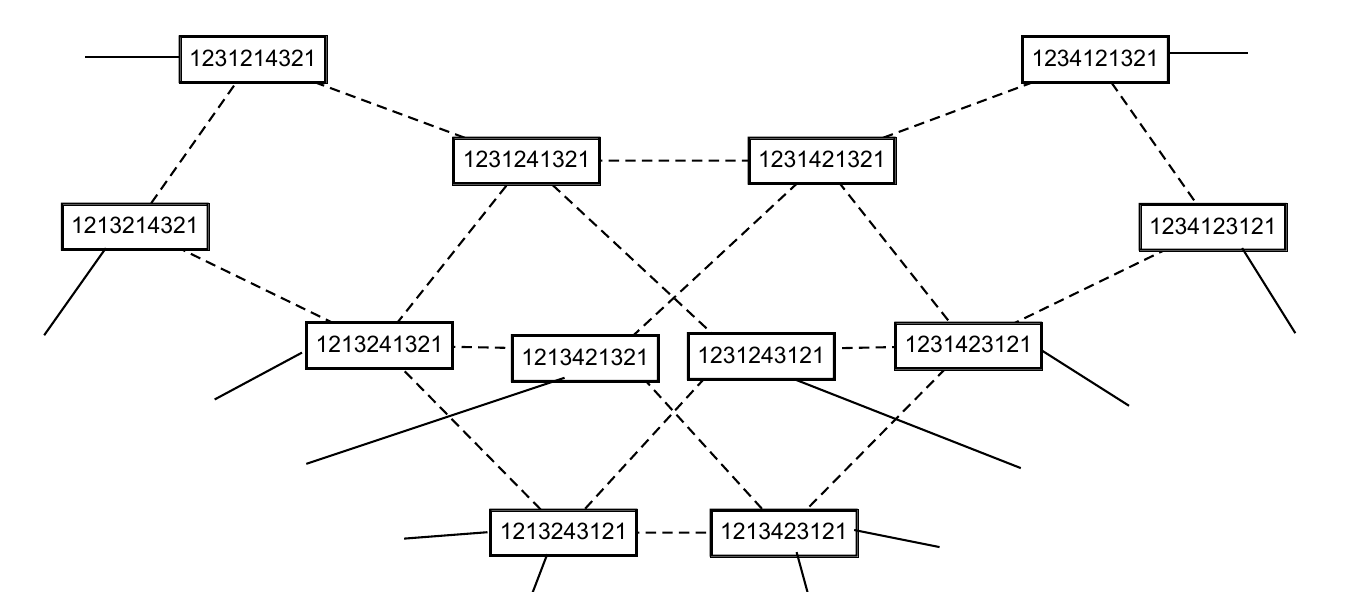}
	\end{center}
\caption{The 12 reduced words in the commutation class $\Red(\mathbf{w_0})$ for $\mathbf{w_0} = (s_1s_2s_3s_4)(s_1s_2s_3)(s_1s_2)(s_1).$  Solid lines denote braid relations leaving the commutation class.}
\label{fig:matsumoto_graph_commutation_class_S5}
\end{figure}

\begin{theorem}
	The expected number of braid moves for a reduced word in the commutation class of the word 
	$\mathbf{w_0} := (s_1 s_2 \cdots s_{n-1})(s_1 s_2 \cdots s_{n-2})\cdots(s_1 s_2)(s_1)$ 
	in $\mathfrak{S}_n$ is one.
\label{thm:braid_moves_commutation_class}
\end{theorem}

We prove Theorem~\ref{thm:braid_moves_commutation_class} by providing a bijection from $\Red(\mathbf{w_0})$ 
to the set of all braid moves in elements of $\Red(\mathbf{w_0})$. 

In a similar spirit to V.~Reiner's translation of Theorem~\ref{thm:braid_moves_reiner} to a statement on standard 
tableaux, in Section~\ref{section.heaps} we use X.~Viennot's theory of heaps~\cite{viennot.1986}
to rephrase Theorem~\ref{thm:braid_moves_commutation_class} as a statement on \defn{shifted tableaux}.  This 
bijection is illustrated in Figure~\ref{fig:matsumoto_graph_shifted_tableaux_S5}.
We define a \defn{braid hook} to be a collection of three boundary cells arranged in the shifted shape $(2,1)$, labeled 
by consecutive numbers $i-1,i,i+1$ (see Definition~\ref{definition.braid hook}).  In this language, 
Theorem~\ref{thm:braid_moves_commutation_class} becomes the following statement.

\begin{theorem}
	The expected number of braid hooks in a shifted SYT of staircase shape is one.
\label{thm:braid_hooks}
\end{theorem}

\begin{figure}[htbp]
	\begin{center}
		\includegraphics[width=\textwidth]{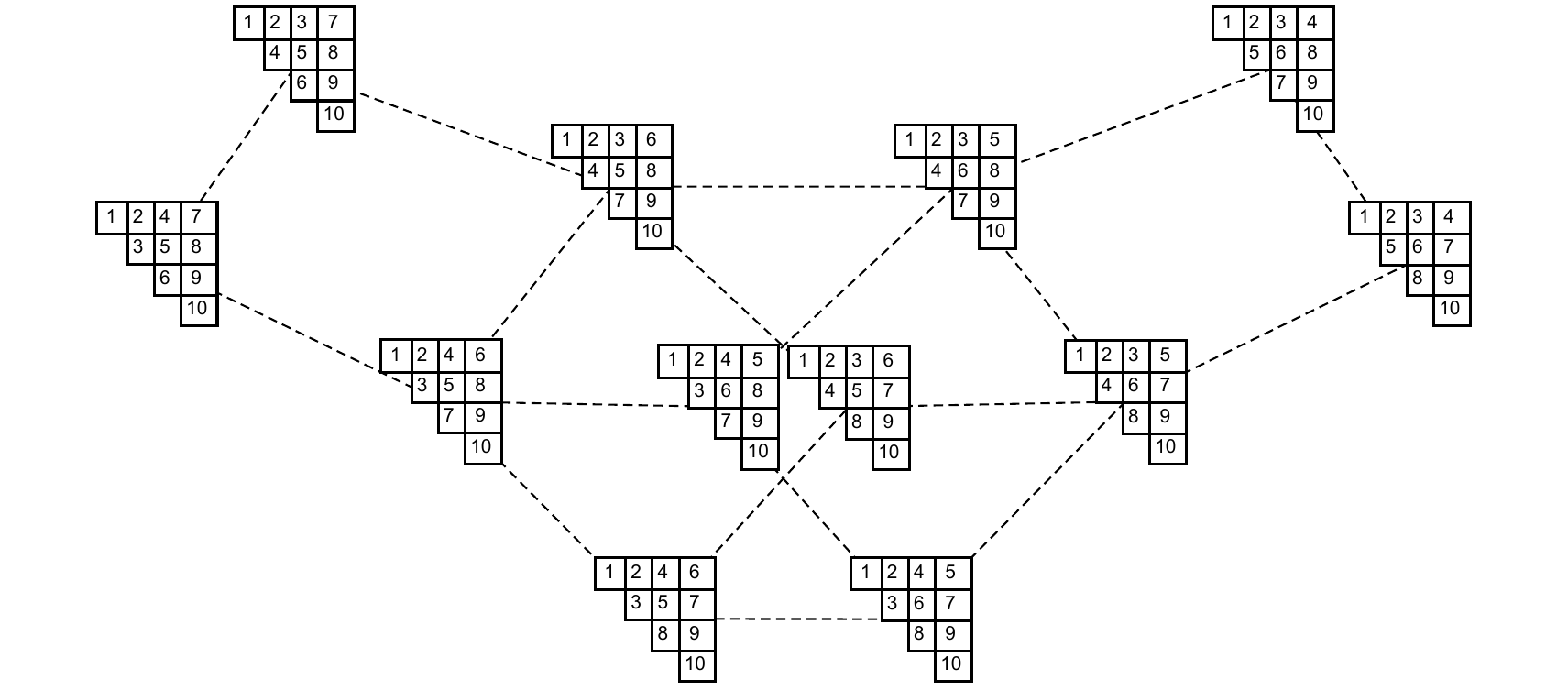}
	\end{center}
\caption{The 12 shifted SYT of staircase shape $(4,3,2,1)$.  Compare with Figure~\ref{fig:matsumoto_graph_commutation_class_S5}.}
\label{fig:matsumoto_graph_shifted_tableaux_S5}
\end{figure}

In fact, we conclude Theorem~\ref{thm:braid_hooks} as a corollary of the much more general 
Theorem~\ref{thm:braid_moves_young_tableaux}, which applies to a certain class of \defn{right-justified} tableaux 
that contain the shifted staircases as a special case.

\subsection{Half-Right-Justified Tableaux}
Recall that the \defn{hyperoctahedral group} $B_n$ is the group generated by $\{s_i\}_{i=1}^{n-1}$ 
(where now $s_i = (i,i+1)(-i,-i-1)$), along with the generator $s_0 := (1,-1)$.  In addition to commutations and 
braid moves, the hyperoctahedral group also satisfies the \defn{long braid move}
\[ 
	s_0 s_1 s_0 s_1 = s_1 s_0 s_1 s_0\;.
\]
Elements in $B_n$ can be represented as signed permutations. Note that the reduced words of the signed
permutation $(-(n-1),-(n-2),\ldots, -2,-1) \in B_{n-1}$ are precisely the same (up to a shift by 1) as the reduced words
in the commutation class of $\mathbf{w_0}$ in Theorem~\ref{thm:braid_moves_commutation_class} (this follows
from Lemma~\ref{lemma.up down braid} below).

M.~Haiman~\cite{haiman.1992} proved that reduced words for the longest element $w_0$ in type $B_n$ are 
equinumerous with SYT of shifted trapezoidal shape.  W.~Kra\'skiewicz~\cite{kraskiewicz.1989} gave an explicit 
insertion procedure, which was used by S.~Billey and T.K.~Lam~\cite{billey_lam.1998} to give an interpretation 
in terms of pattern avoidance and a link to Stanley symmetric functions. Similarly to the case of $\mathfrak{S}_n$ 
and SYT of staircase shape, promotion on shifted trapezoids corresponds to the action
\[  
	s_{i_1}s_{i_2}\cdots s_{i_N} \mapsto s_{i_2}\cdots s_{i_N}s_{i_1}
\]
on $\Red(w_0)$, where $N$ is now the length of the longest element $w_0$ in type $B_n$.  
Using this technology and a similar method to that in~\cite{reiner.2005}, B.~Tenner~\cite{tenner.2007}
proved a type $B$ analogue of Theorem~\ref{thm:braid_moves_reiner}.

\begin{theorem}[B.~Tenner~\cite{tenner.2007}]
	The expected number of braid moves in $\Red(w_0)$ in type $B_n$ is $2-4/n$.  The expected number 
	of long braid moves is $\frac{2}{n^2-2}$.
\label{thm:tenner}
\end{theorem}

By considering \defn{half-right-justified tableaux}, which are certain tableaux that can be paired with themselves to 
produce right-justified tableaux (see Figure~\ref{figure.half right}) and which include shifted trapezoidal shapes,
we provide a complementary result to Theorem~\ref{thm:tenner} in Section~\ref{sec:half_right_justified}.

\begin{theorem}
	 The expected number of braid hooks in a shifted SYT of trapezoidal shape is one half.
\label{thm:braid_moves_half_young_tableaux}
\end{theorem}

We are not aware of an interpretation of braid hooks in shifted SYT of trapezoidal shapes in terms of the corresponding
reduced words.

\subsection{Homomesy}

In Section~\ref{sec:even_odd_homomesy}, we provide an independent bijective proof of 
Theorem~\ref{thm:braid_hooks} and its generalization Theorem~\ref{thm:braid_moves_young_tableaux} 
by refining the previous statements using \defn{homomesy}.
Homomesy was introduced by Panyushev~\cite{panyushev.2009} and later Propp and
Roby~\cite{propp.roby.2013}. It involves partitioning the underlying set into orbits 
under some group action, and proving that the
averaging property still holds on each orbit. Formally, let $S$ be a
set, $s$ a statistic on $S$, and $G$ a group acting on $S$. Then $s$ is
\defn{homomesic} with respect to the action of $G$ if the average of
$s$ on orbits is constant.

In our case, $G$ is the dihedral group generated by a ``bipartite'' version of
promotion, namely the odd and even operators $\tau_o$ and $\tau_e$.
\begin{theorem}
  \label{thm:braid_hooks_homomesy_even_odd}
  The number of braid hooks is homomesic with respect to the action of the
  group $\langle \tau_o,\tau_e\rangle$ on shifted SYT of staircase shape.
\end{theorem}
This statement admits the same generalization, stated in Theorem~\ref{theorem.homomesy},
to right justified tableaux as in Theorem~\ref{thm:braid_moves_young_tableaux}.
Section~\ref{sec:even_odd_homomesy} provides a self-contained bijective proof
of Theorem~\ref{theorem.homomesy words} (which is a reformulation of 
Theorem~\ref{theorem.homomesy}) in the terms of reduced words. It is similar in spirit
to the proof of Theorem~\ref{thm:braid_moves_young_tableaux}  in that it uses certain toggle
operators which admit inverses when a braid is present.

It turns out that in general the number of braid hooks is \emph{not} homomesic with respect to 
the abelian subgroups of our dihedral group $\langle \tau_o, \tau_e \rangle$. Hence 
Theorem~\ref{thm:braid_hooks_homomesy_even_odd} provides an example 
of homomesy with respect to a nonabelian group that is not implied by a homomesy of an 
abelian subgroup (see~\cite[Section~2]{roby.2015} for a discussion about this).
In fact this is one of the very first examples of dihedral homomesy.
To the best of our knowledge, the only other known examples have appeared 
in~\cite{hohlweg.lortie.raymond.2010}, where it is proven that the barycenter of any associahedron
coincides with that of the permutahedron by using homomesy with respect to
dihedral subgroups, and in~\cite{pilaud.stump.2015}, where this statement is
generalized to Coxeter groups.

In Section~\ref{subsection.homomesy poset} we give a homomesy result for more general posets,
where the statistic is given by descents.

\begin{openproblem}
It would be interesting to extend the methods developed in this paper to the full set of reduced words
for $w_0$, that is, to study Reiner's original problem~\cite{reiner.2005} with these new techniques.
\end{openproblem}

\section*{Acknowledgements}

This project began in March 2015 at the workshop ``Dynamical algebraic combinatorics'' at the  American Institute 
of Mathematics (AIM).  We are indebted to Z.~Hamaker and V.~Reiner, who were part of our working group at AIM, for
many invaluable discussions and suggestions, and to H.~Thomas for his suggestions regarding
Section~\ref{subsection.homomesy poset}! We thank AIM for financial support and a stimulating environment for
collaboration, and the other organizers J.~Propp, T.~Roby, and J.~Striker for helping to organize the event.
The last author would like to thank G.~Panova for useful conversations. We also thank the anonymous referee
for helpful comments.

AS is partially supported by NSF grants OCI--1147247 and DMS--1500050,
and a short visit to Orsay was partially funded by DIGITEO/GT STIC
No 2015-XXD.

This research was driven by computer exploration using the open-source
mathematical software \texttt{Sage}~\cite{sage} and its algebraic
combinatorics features developed by the \texttt{Sage-Combinat}
community~\cite{Sage-Combinat}.

\section{Reduced words and Heaps}
\label{section.heaps}

In this section, we explain the bijection between reduced words in the commutation class of
$\mathbf{w_0} := (s_1 s_2 \cdots s_{n-1})(s_1 s_2 \cdots
s_{n-2})\cdots(s_1 s_2)(s_1)$ and shifted standard staircase tableaux.
This uses X.~Viennot's \defn{heap model}~\cite{viennot.1986} to construct a poset whose linear extensions are
in bijection with the reduced words in the commutation class. The linear extensions of the poset
can then be interpreted as tableaux.

To construct the poset for a reduced word $\mathbf{w} = s_{i_\ell} \cdots s_{i_1}$ of $w\in \mathfrak{S}_n$, associate 
a column to each simple transposition $s_i$ ($1\le i<n$) of $\mathfrak{S}_n$. We order the columns from left
to right with increasing $i$, so that the column for $s_i$ is adjacent to the columns of $s_{i-1}$ and $s_{i+1}$ 
(whenever they exist). Starting with the rightmost generator $s_{i_1}$ in $\mathbf{w}$ and moving left generator 
by generator in $\mathbf{w}$, successively drop a ``heap'' in column $i$ for each $s_i$ encountered. These heaps 
are wide enough such that two heaps in adjacent columns overlap. Note that a heap gets stuck above another heap
when the two heaps are in adjacent columns, which coincides with the case that the corresponding simple transpositions 
do not commute. The vertices of the poset $P_{\mathbf{w}}$ are precisely the heaps, and the covering relations are 
given by $v_2 \lessdot v_1$ if and only if $v_2$ is the lowest vertex above $v_1$ in a column adjacent to $v_1$.

\begin{example}
  Figure~\ref{figure.heaps} shows the construction of the poset
  $P_{\mathbf{w}}$ and its linear extension for three reduced words in
  $\Red(w_0)$ in $\mathfrak{S}_5$. The first word has no particular significance, the
  second word is $\mathbf{w_0}$, and the third one is in the commutation class
  $\Red(\mathbf{w_0})$.
\end{example}

\begin{figure}[t]
\newcommand{\fig}[1]{\includegraphics[height=3.5cm,trim=2cm 0cm 2cm 0cm,clip]{#1}}
\newcommand{\separator}{}
\fig{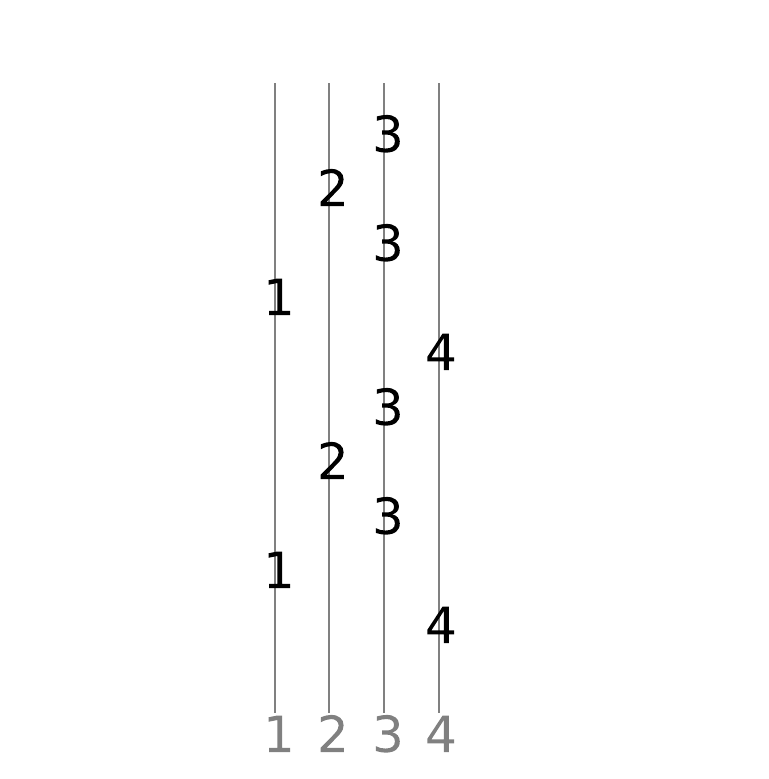}\separator
\fig{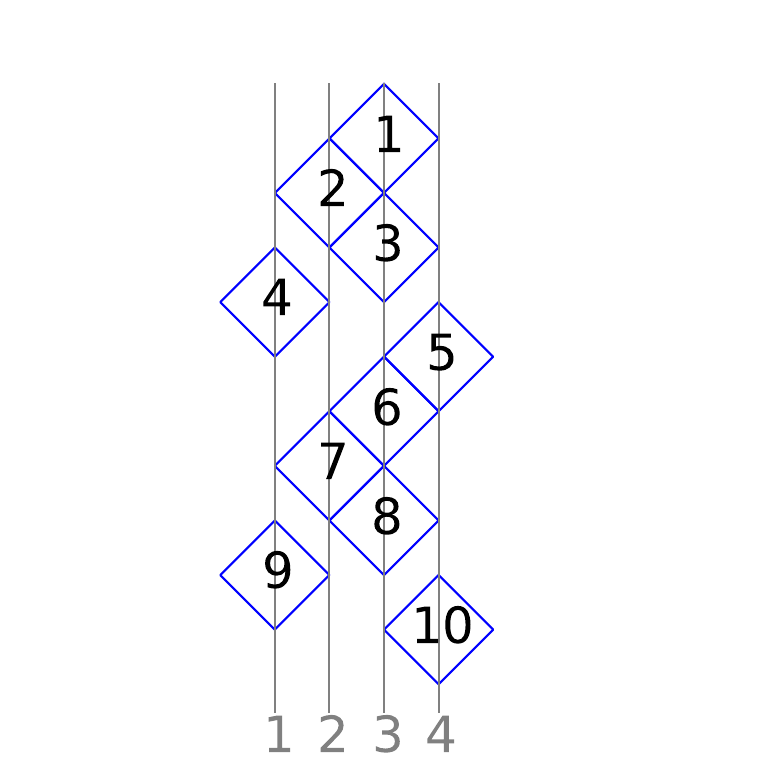}\separator
\fig{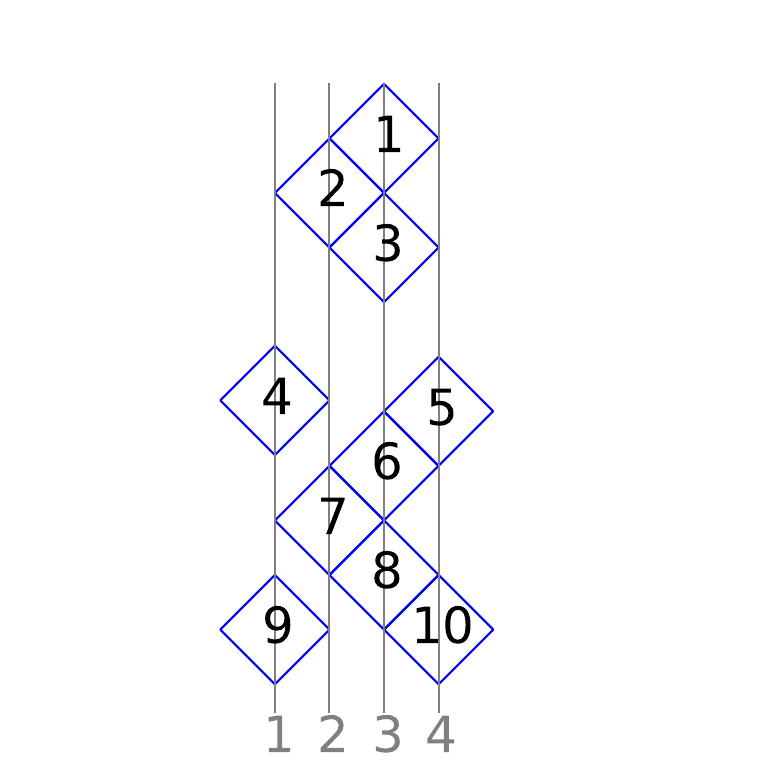}\separator
\fig{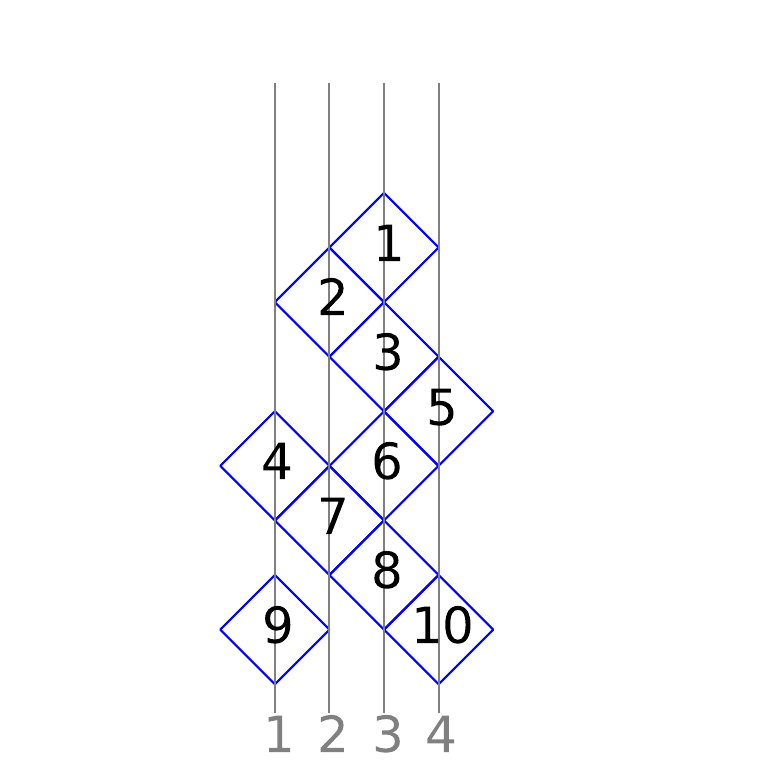}\separator
\phantom{\fig{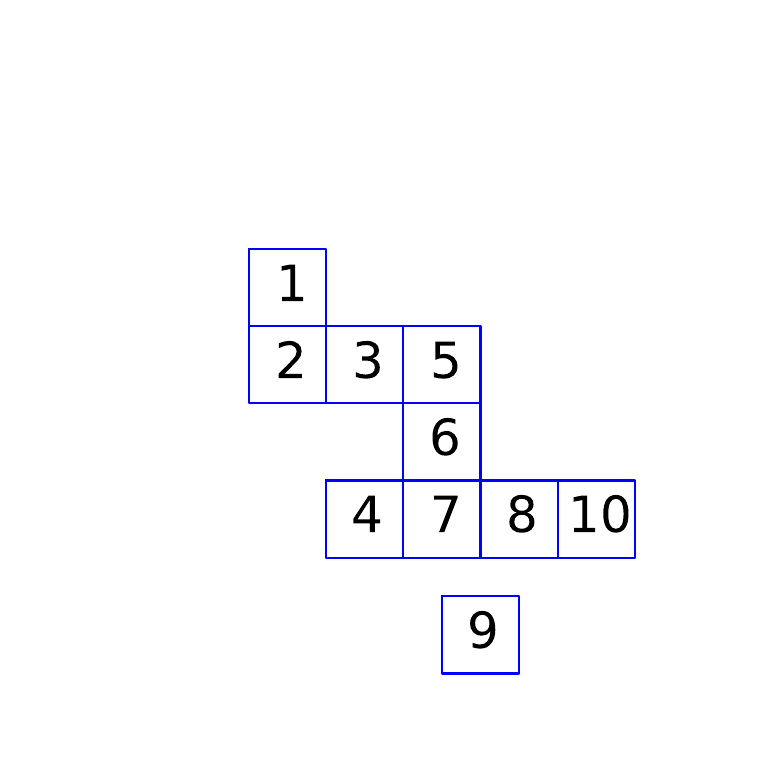}}

\fig{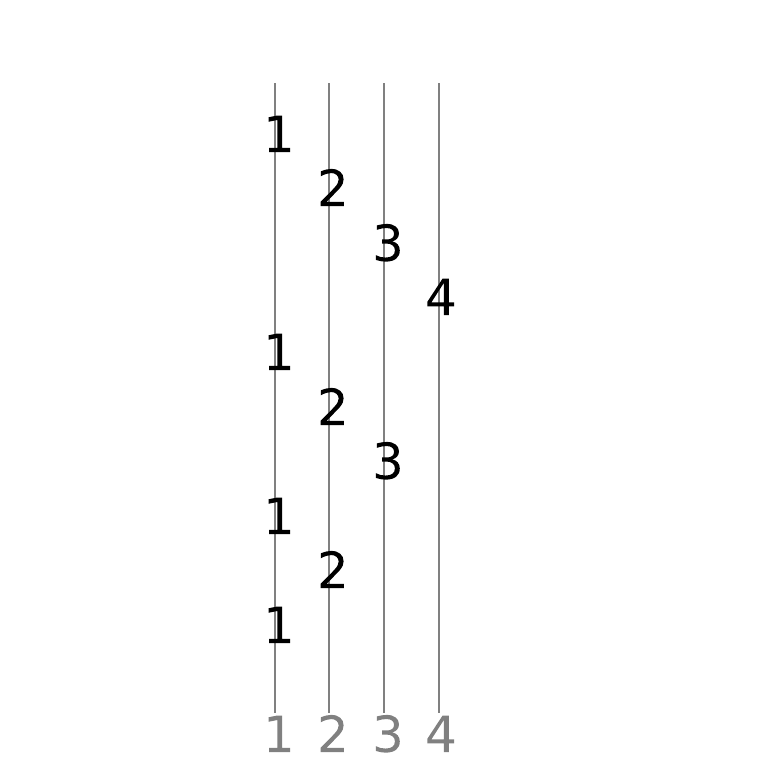}\separator
\fig{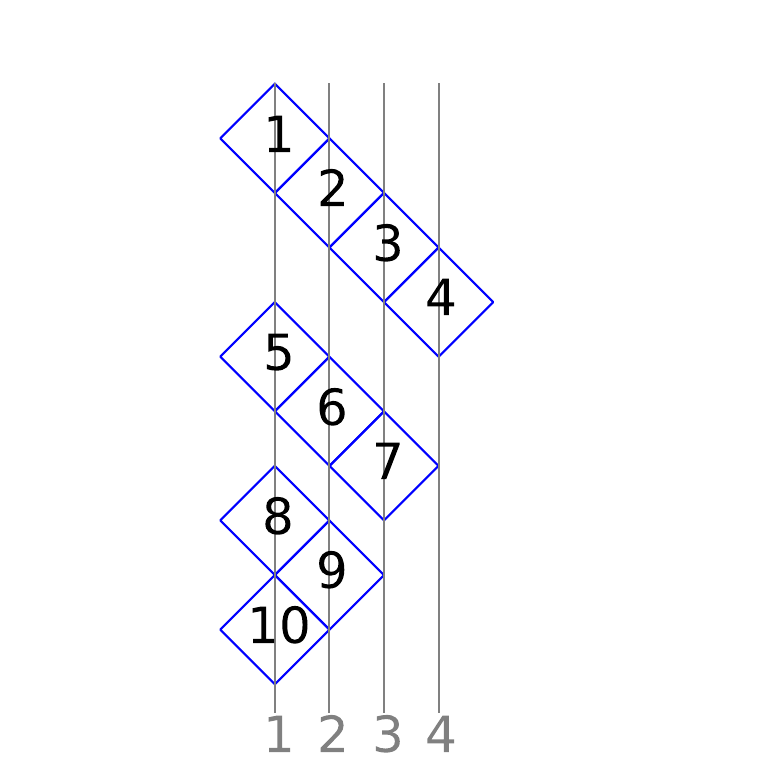}\separator
\fig{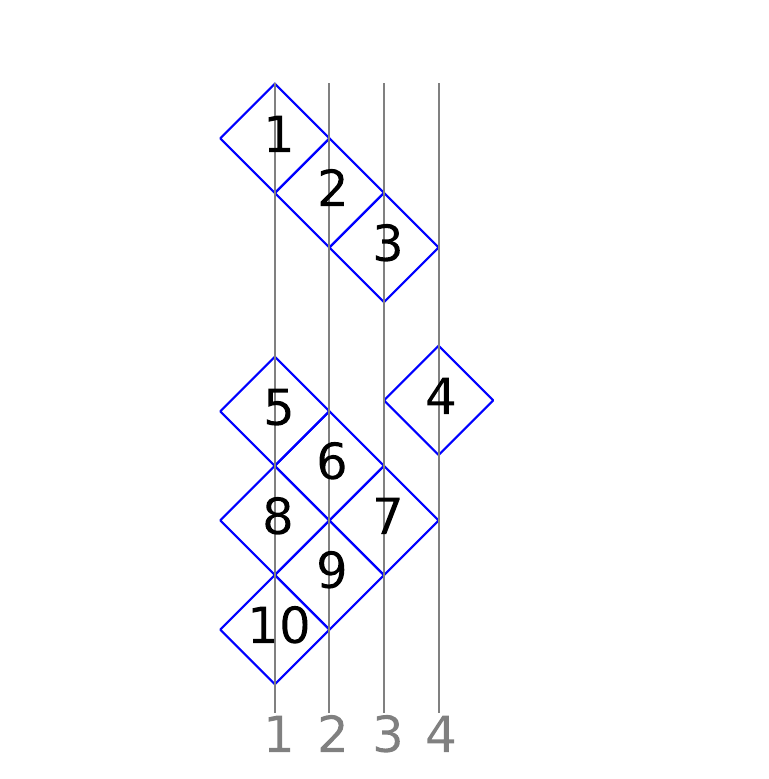}\separator
\fig{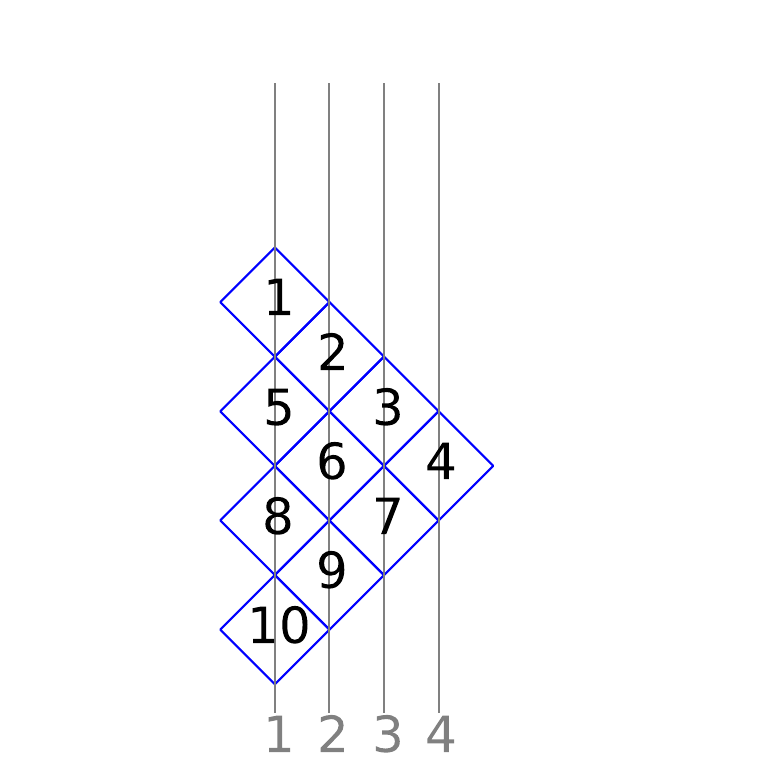}\separator
\fig{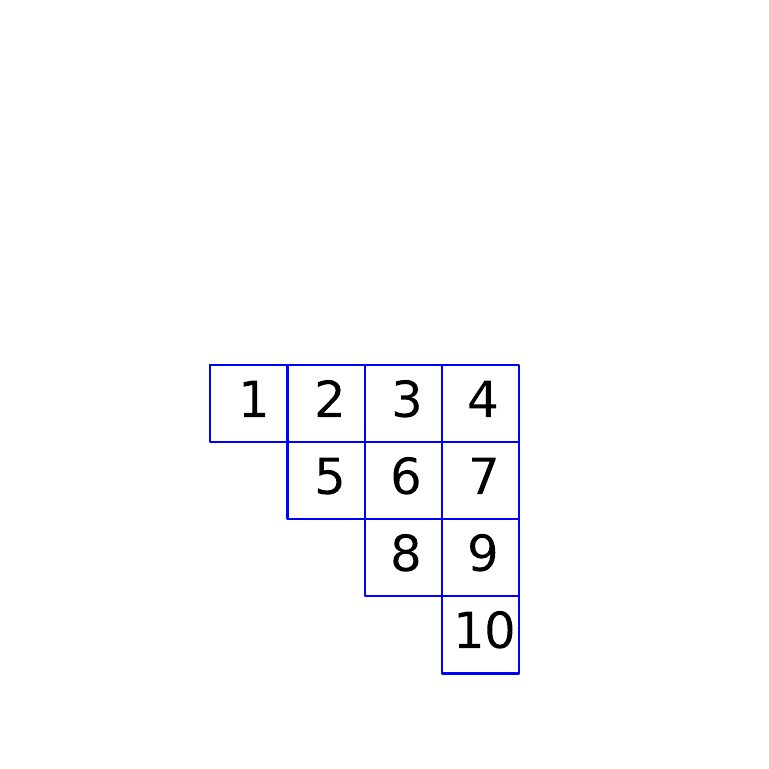}

\fig{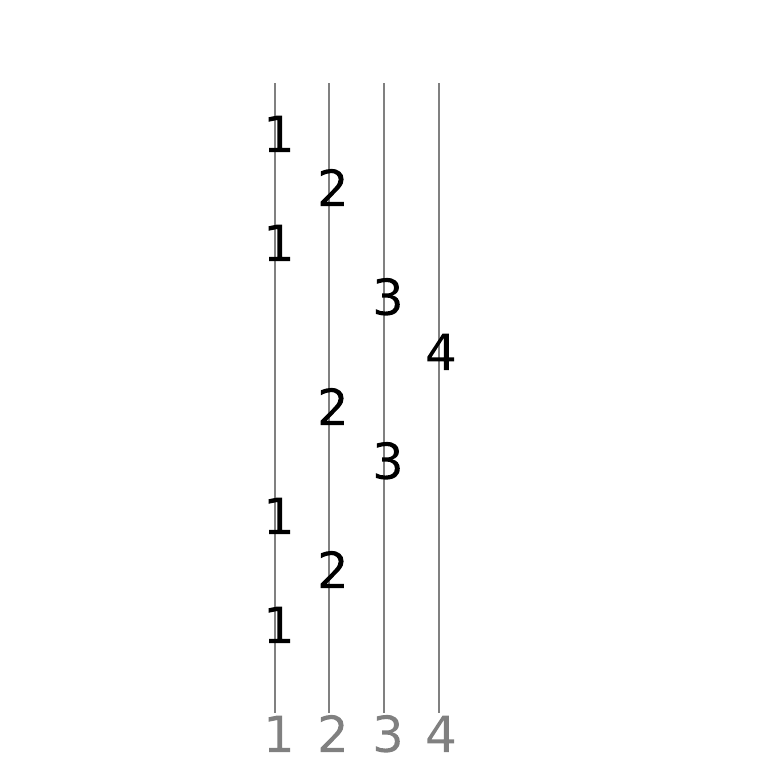}\separator
\fig{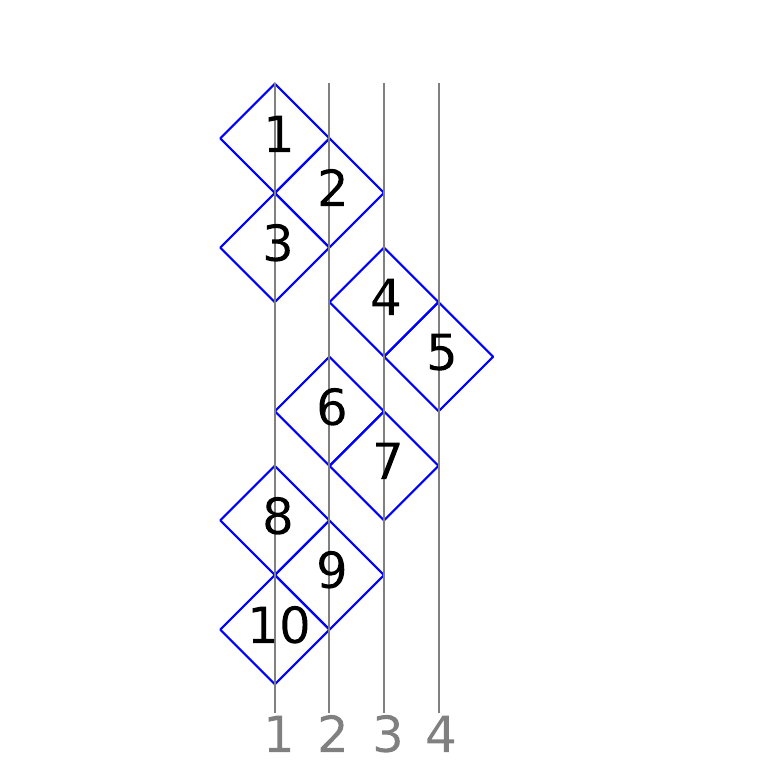}\separator
\fig{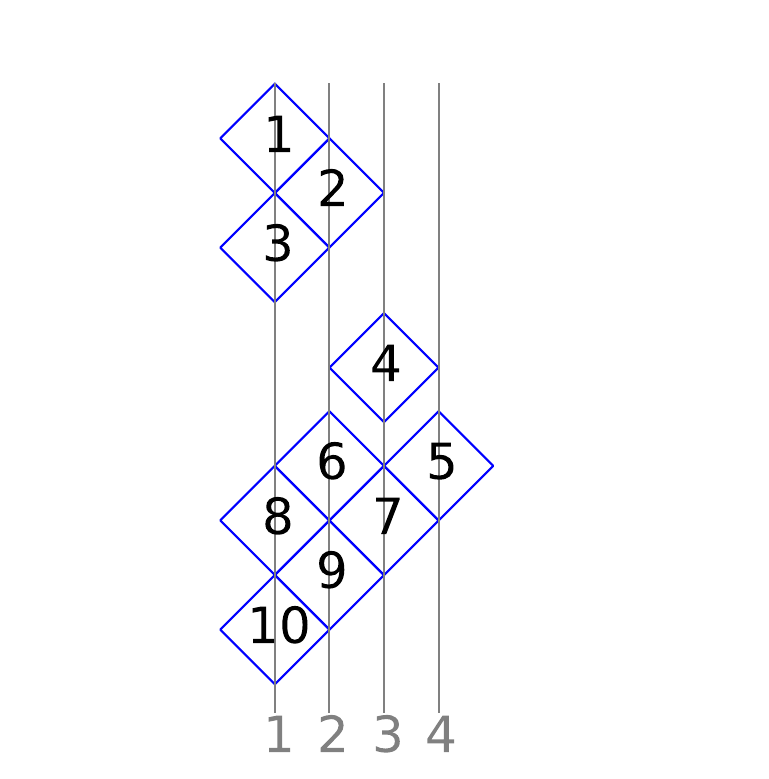}\separator
\fig{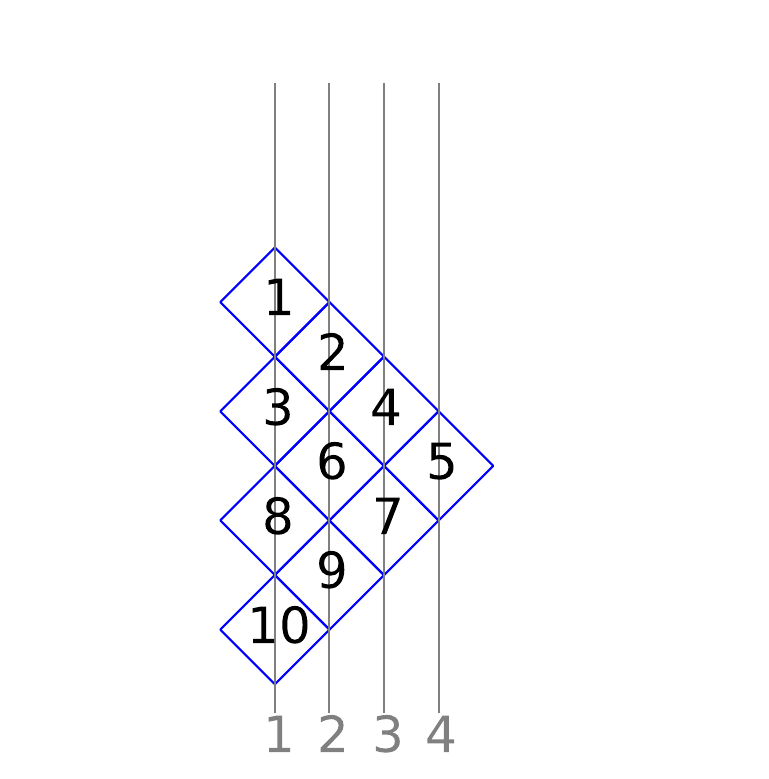}\separator
\fig{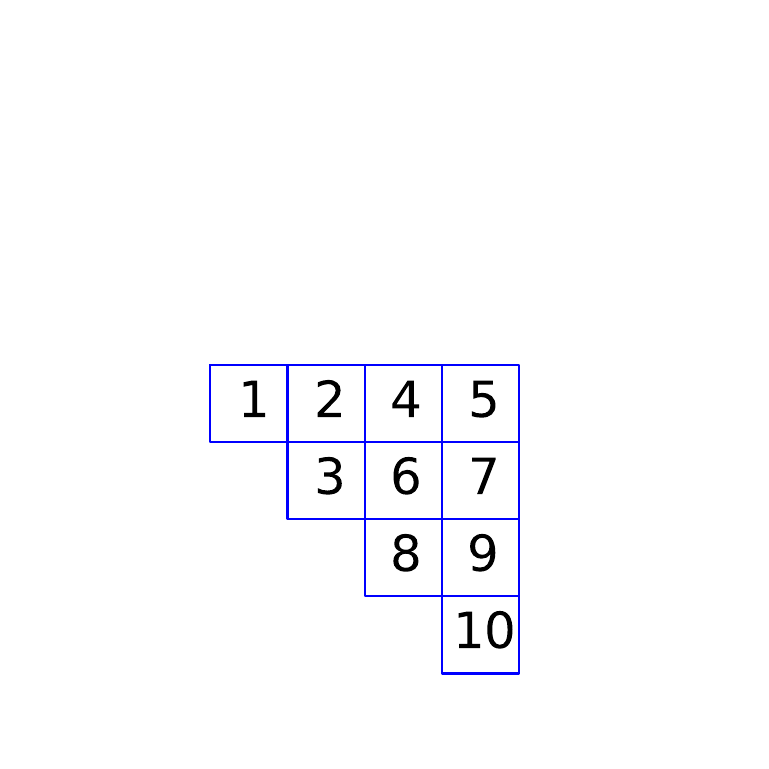}

\caption{Incrementally building the heap posets (and shifted staircase
  tableaux when relevant) for the reduced words $s_3 s_2 s_3 s_1 s_4 s_3 s_2 s_3 s_1 s_4$, 
  $(s_1 s_2 s_3 s_4)(s_1 s_2 s_3)( s_1 s_2)(s_1)$, and
  $s_1 s_2 s_1 s_3 s_4 s_2 s_3 s_1 s_2 s_1$ of $w_0$ in $\mathfrak{S}_5$.
  \label{figure.heaps}}
\end{figure}

As suggested by the above example, any reduced word in the same
commutation class as $\mathbf{w}$ yields the same poset
$P_{\mathbf{w}}$. In fact, keeping track of the order in which each
heap (or vertex) is added gives a linear extension of this poset; it is not 
hard to see that the elements of $\Red(\mathbf{w})$ are in
bijection with such linear extensions.

Let $\mathbf{w}$ be any reduced word in the commutation class of
$\mathbf{w_0}$. The poset $P_{\mathbf{w}}$ has $\Delta_n:=(n,n-1,\ldots,1)$ elements on the
NE-diagonals. Rotating this poset (resp. linear extension of the poset)
counterclockwise by $45^o$ yields a \defn{shifted staircase partition}
(resp. \defn{standard shifted staircase tableau}). A shifted staircase tableau is characterized as increasing along rows from left to right and increasing along columns
from top to bottom. We denote the set of all standard shifted staircase tableaux
of shape $\Delta_n$ by $\ShSYT(\Delta_n)$.

From the bijection
\[
	\nu \colon \Red(\mathbf{w}_0) \to \ShSYT(\Delta_n)
\] 
we obtain the following result.
\begin{lemma}
\label{lemma.up down braid}
The only possible braid moves in elements of $\Red(\mathbf{w}_0)$ are those of the form $s_1 s_2 s_1$.
\end{lemma}

\begin{proof}
Observe that under the bijection $\nu$, braid moves $s_i s_{i+1} s_i$ and $s_{i+1} s_i s_{i+1}$
in a reduced word would result in hooks in the corresponding tableau of the form
\begin{equation}
\label{equation.hooks}
	\scalebox{.7}{\tableau[mbY]{k-1 &k\\ \bl &k+1}}
	\qquad \text{and} \qquad 
	\scalebox{.7}{\tableau[mbY]{k-1& \bl\\ k&k+1}}
\end{equation}
respectively. Note that the first hook can sit on the diagonal, whereas the second hook
has to appear inside the tableau. If the hook appears inside the tableau, there is a letter $a$ in
\[
	\scalebox{.7}{\tableau[mbY]{k-1 &k\\ a &k+1}}
	\qquad \text{or} \qquad 
	\scalebox{.7}{\tableau[mbY]{k-1& a\\ k&k+1}}
\]
such that $k-1<a<k+1$ by the tableau conditions. This implies that $a=k$, which contradicts the
fact that the tableau is standard and $k$ already appears. Hence the only possibility is for the first hook
in~\eqref{equation.hooks} to appear on the diagonal. Under the bijection $\nu$ this corresponds precisely to
a braid move $s_1 s_2 s_1$.
\end{proof}

\begin{definition}
\label{definition.braid hook}
Let $t \in \ShSYT(\Delta_n)$. Then we say that $k$ is a \defn{braid hook} of $t$ if there is a sequence of 
consecutive letters $k-1,k,k+1$ in $t$ with no box below the box containing 
$k-1$, as in the first picture in~\eqref{equation.hooks}.
\end{definition}

\begin{example}
The following tableau in $\ShSYT(\Delta_6)$
\[
	\tableau[sY]{1&2&3&7&9&14\\ \bl&\darkblue{4}&\darkblue{5}&8&12&15\\ \bl&\bl&\darkblue{6}&10&13&17\\
	\bl&\bl&\bl&11&16&18\\ \bl&\bl&\bl&\bl&\darkred{19}&\darkred{20}\\ \bl&\bl&\bl&\bl&\bl&\darkred{21}}
\]
has braid hooks $5$ (involving the letters $4,5,6$ on the second position of the diagonal) and
$20$ (involving $19,20,21$ on the fifth position of the diagonal).
\end{example}

\noindent
By the results of this section, and using the bijection $\nu$, we can reformulate
Theorem~\ref{thm:braid_moves_commutation_class} entirely in terms of tableaux.

{
\renewcommand{\thetheorem}{\ref{thm:braid_hooks}}
\begin{theorem}
	 The expected number of braid hooks of elements in  $\ShSYT(\Delta_n)$ is one.
\end{theorem}
\addtocounter{theorem}{-1}
}
Theorem~\ref{thm:braid_hooks} (and, as a corollary,
Theorem~\ref{thm:braid_moves_commutation_class}) will be proved in the
next section in a more general setting.

\section{Right-justified tableaux}
\label{section.right-justified}

The statement of Theorem~\ref{thm:braid_hooks} regarding the expected number of braid hooks in standard 
shifted tableaux of staircase shape can be generalized to more general shapes. Let 
$\lambda = (\lambda_1,\lambda_2,\ldots,\lambda_\ell)$ be a partition, which means that $\lambda_1,\ldots,\lambda_\ell$
are integers satisfying $\lambda_1\ge \lambda_2 \ge \cdots \ge \lambda_\ell \ge 0$. We define $\rSYT(\lambda)$ to be the 
set of standard tableaux of the diagram given by $\lambda$, where we right-justify all rows. This definition requires 
as usual that all rows and columns are strictly increasing from left to right and top to bottom.
Note that
\begin{equation}
\label{equation.rSYT=shSYT}
	\rSYT(\Delta_n)=\ShSYT(\Delta_n)\;.
\end{equation}
A braid hook for $t\in \rSYT(\lambda)$ is defined in the same way as in Definition~\ref{definition.braid hook}.

\begin{example}
Let $\lambda=(5,2,1)$. Then
\[
	\rSYT(\lambda) = \left\{ \;
	\tableau[sY]{1&2&3&4&5\\ \bl&\bl&\bl&\darkblue{6}&\darkblue{7}\\ \bl&\bl&\bl&\bl&\darkblue{8}} \;,\;
	\tableau[sY]{1&2&\darkblue{3}&\darkblue{4}&6\\ \bl&\bl&\bl&\darkblue{5}&7\\ \bl&\bl&\bl&\bl&8}\;
	\right\},
\]
where the braid hooks are indicated in blue. Note that the expected number of braid hooks is one in this case.
\end{example}

\begin{theorem}
	Let $\lambda = (\lambda_1,\lambda_2,\ldots,\lambda_\ell)$ be a partition such that 
	$\lambda_1 > \lambda_2$ and $\lambda_\ell = 1$.  
	Then the expected number of braid hooks in $\rSYT(\lambda)$ is one.
\label{thm:braid_moves_young_tableaux}
\end{theorem}

Note that by~\eqref{equation.rSYT=shSYT}, Theorem~\ref{thm:braid_hooks} is the special case of
Theorem~\ref{thm:braid_moves_young_tableaux} for $\lambda=\Delta_n$.
In this section we prove the latter by constructing a bijection
\begin{equation}
\label{equation.varphi}
	\varphi \colon \{(k,t) \mid \text{$k$ a braid hook in $t\in \rSYT(\lambda)$}\} 
	\to \rSYT(\lambda)\;.
\end{equation}
The map $\varphi$ is defined using certain operators akin to the
promotion operator on tableaux. For $1\le i< |\lambda|$, let
\begin{equation}
\label{equation.tau}
\begin{split}
	\tau_i \colon \rSYT(\lambda) &\to \rSYT(\lambda)\\
	t & \mapsto t. \tau_i
\end{split}
\end{equation}	
be the map that interchanges $i$ and $i+1$ in $t$ if the result is again in $\rSYT(\lambda)$ 
and otherwise leaves $t$ fixed. Define
\[
	\varphi(k,t) := t.\partial^*_k \partial_k\;,
\]
where $\partial_k := \tau_k \tau_{k+1} \cdots \tau_{|\lambda|-1}$ and
$\partial^*_k := \tau_{k-1} \tau_{k-2} \cdots \tau_1$.  Note that the
operators $\partial_k$ and $\partial^*_k$ are partial \defn{promotion} and 
\defn{inverse promotion} operators, respectively. For example, as
explained in~\cite{stanley.2009}, the operator
\[
	\partial = \partial_1 = \tau_1 \tau_2 \cdots \tau_{|\lambda|-1}
\]
coincides with M.P.~Sch\"utzenberger's \defn{promotion} on tableaux.  This promotion operator is more commonly 
defined using \defn{jeu-de-taquin} as follows:
given a tableau, remove the letter 1 and successively slide the smaller of the right and lower neighbor cells
(if they exist) into the empty slot, until the empty slot occupies a cell with no nonempty right or lower neighbor cells.   Now enter $|\lambda|+1$ into the empty cell and subtract one from each entry.
Similarly, the inverse promotion operator $\partial^*=\partial^{-1}$ can be defined using a sliding algorithm
starting from the largest letter in the tableau. The inverse promotion operator may be expressed as 
\[
	\partial^* = \partial^*_{|\lambda|} = \tau_{|\lambda|-1}\tau_{|\lambda|-2}\cdots\tau_1\;.
\]

The sequence of empty slots in the jeu-de-taquin formulation of the promotion operator define the 
\defn{promotion sliding path}, denoted $\mathcal{L}$. 
The \defn{inverse promotion sliding path} is denoted by $\mathcal{R}$. 
Their description might give the impression that $\mathcal{L}$ and $\mathcal{R}$ are oppositely 
directed (since $\mathcal{L}$ is defined by removing the letter 1 and then sliding into the empty slot, whereas
for $\mathcal{R}$ one removes $|\lambda|$). However, we define them only as undirected paths. 
Later in this section we will treat them both as paths  directed from the top left to bottom right.

\begin{example}
\label{example.sliding paths}
We illustrate the promotion sliding path $\mathcal{L}$ by bold cells and the inverse promotion path
$\mathcal{R}$ by shaded blue cells:
\[
	\mathcal{L}:
	\tableau[sY]{\tf 1&\tf 2&4&6&10&12\\
	\bl& \tf 3&\tf 5&\tf 7&11&13\\
	\bl&\bl&8&\tf 9& \tf 14& 17\\
	\bl&\bl&\bl&15&\tf 16& \tf 18\\
	\bl&\bl&\bl&\bl&19&\tf 20\\
	\bl&\bl&\bl&\bl&\bl&\tf 21}
	\qquad \qquad \qquad 
	\mathcal{R}:
	\tableau[sY]{\lb \overlay 1&\lb \overlay 2&\lb \overlay 4&6&10&12\\
	\bl& 3&\lb \overlay5&7&11&13\\
	\bl&\bl&\lb \overlay 8&\lb \overlay 9& 14& 17\\
	\bl&\bl&\bl&\lb \overlay 15&\lb \overlay 16&18\\
	\bl&\bl&\bl&\bl&\lb \overlay 19&\lb \overlay 20\\
	\bl&\bl&\bl&\bl&\bl&\lb \overlay 21}\;.
\]
\end{example}
Throughout this section, we will continue to illustrate the promotion path $\mathcal{L}$ and inverse promotion path $\mathcal{R}$ with bold and shaded cells, respectively.

\begin{lemma}
\label{lemma.phi bijection}	
	Let $\lambda = (\lambda_1,\lambda_2,\ldots,\lambda_\ell)$ be a partition such that 
	$\lambda_1 > \lambda_2$ and $\lambda_\ell = 1$.  
	Then $\varphi$ is a bijection.
\end{lemma}

\begin{proof}
  To show that $\varphi$ is a bijection, we explicitly construct its
  inverse. To this end, let $t\in \rSYT(\lambda)$. We want to
  associate to $t$ a pair $(k,t')$, where $k$ is a braid hook in
  $t' \in \rSYT(\lambda)$ and $t=\varphi(k,t')$. Given that each
  $\tau_i$ is a bijection, so is $\partial^*_k\partial_k$, and
  $t'=t.(\partial^*_k\partial_k)^{-1}$ is completely determined by
  $k$. Hence, all that is needed is to prove is that there exists a
  unique $k$ such that $k$ is a braid hook of
  $t.(\partial^*_k\partial_k)^{-1}$.

To achieve this, we use that $\partial_k$ and $\partial^*_k$ are
the partial promotion and inverse promotion operator, respectively, and
study the crossings of the promotion path $\mathcal{L}$ and the
inverse promotion path $\mathcal{R}$ in $t$.
Namely, note that $k$ is a braid hook of $t'=t.(\partial^*_k\partial_k)^{-1}$
if and only if the promotion path $\mathcal{L}$ and inverse promotion path $\mathcal{R}$
of $t$ cross in the left inner corner specified by $k$ according to the following
configuration
\begin{equation}
\label{equation.allowed}
   \tableau[sY]{\tf $x$ & \tf \lb \overlay $k$\\ \bl & \lb \overlay $y$}\;,
\end{equation}
where $x$ and $y$ are any allowed values.
This can be seen as follows: the action of  
$\partial_k^{-1}=\tau_{|\lambda|-1}\cdots\tau_k$ on $t'$ performs
jeu-de-taquin along the suffix of the inverse promotion path $\mathcal{R}$,
down to value $k$. At the end, $y$ is replaced by $k+1$ if and only
if $k$ moves into the cell of $y$ under jeu-de-taquin, that is, if the inverse promotion path 
$\mathcal{R}$ of $t$ is as in~\eqref{equation.allowed}. The same
reasoning relates the replacement of $x$ by $k-1$ in $t'$ with the
position of the promotion path $\mathcal{L}$ of $t$ as in~\eqref{equation.allowed}.

It remains to prove that the paths $\mathcal{L}$ and $\mathcal{R}$ of $t$ admit exactly one such crossing.
First notice that the $2\times 2$ configuration
\begin{equation}
\label{equation.forbidden1}
	\tableau[sY]{\tf $x$ & b \\ \tf \lb \overlay $a$& \lb \overlay $y$}
\end{equation}
is forbidden in $t$. Namely, the conditions for $\mathcal{L}$ impose that
$a<b$ whereas the conditions for $\mathcal{R}$ require that $a>b$, a contradiction.
Symmetrically, the following $2\times 2$ configuration is forbidden:
\begin{equation}
\label{equation.forbidden2}
	\tableau[sY]{\tf $x$ & \tf \lb \overlay $b$\\ $a$ & \lb \overlay $y$}\;.
\end{equation}
If the letter $a$ below $x$ is missing, however, then~\eqref{equation.forbidden2} \emph{is} allowed, and this recovers configuration~\eqref{equation.allowed} with $b=k$.

By the conditions on right-justified tableaux, the letter 1 is in the top leftmost cell of $t$ and the largest letter 
$|\lambda|$ is in the bottom rightmost cell of $t$. Hence, both sliding paths $\mathcal{L}$ and $\mathcal{R}$ reach 
from the top leftmost cell to the bottom rightmost cell of $t$. Whenever the two paths overlap on a horizontal step
$\tableau[sY]{\tf \lb \overlay $a$ & \tf \lb \overlay $b$}$, let us consider $\mathcal{R}$ to be (locally) above 
$\mathcal{L}$. If, on the other hand, they overlap on a vertical step 
$\tableau[sY]{\tf \lb \overlay $a$\\ \tf \lb \overlay $b$}$, then we consider $\mathcal{L}$ to be (locally) above $\mathcal{R}$.
If the two paths do not overlap, the northeastern path is considered to be (locally) above the other. 

Notice that the two paths $\mathcal{L}$ and $\mathcal{R}$ overlap in the two top leftmost horizontal cells
since $\lambda_1>\lambda_2$. Likewise for the last two vertical steps in the bottom right corner the paths overlap since 
$\lambda_\ell = 1$. Hence (according to our conventions) the paths start out with $\mathcal{R}$ 
above $\mathcal{L}$, and finish with $\mathcal{L}$ above $\mathcal{R}$. The forbidden 
configurations~\eqref{equation.forbidden1} and~\eqref{equation.forbidden2}
are exactly those that prevent the two paths from crossing from ($\mathcal{R}$ above $\mathcal{L}$) to
($\mathcal{L}$ above $\mathcal{R}$) or vice versa, with a single
exception: configuration~\eqref{equation.allowed} allows for a crossing
from ($\mathcal{R}$ above $\mathcal{L}$) to ($\mathcal{L}$ above
$\mathcal{R}$) on a left inner corner, and corresponds to an instance of a braid hook (indeed, at this position the paths will 
not share any steps, but rather pass orthogonally through one another).
Because of the initial and final conditions, such a crossing must happen exactly once.
\end{proof}

\begin{example}
\label{example.sliding paths1}
Superimposing the two sliding paths of Example~\ref{example.sliding paths}
\[
	t=
	\tableau[sY]{\tf \lb \overlay 1&\tf \lb \overlay 2&\lb \overlay 4&6&10&12\\
	\bl& \tf 3&\tf \lb \overlay5&\tf 7&11&13\\
	\bl&\bl&\lb \overlay 8&\tf \lb \overlay 9& \tf 14& 17\\
	\bl&\bl&\bl&\lb \overlay 15&\tf \lb \overlay 16&\tf 18\\
	\bl&\bl&\bl&\bl&\lb \overlay 19&\tf \lb \overlay 20\\
	\bl&\bl&\bl&\bl&\bl&\tf \lb \overlay 21}
\]
one notices that there is precisely one configuration of the form~\eqref{equation.allowed}, namely
with $x=3$, $b=5$ and $y=8$. Hence $\varphi^{-1}(t)$ is the braid $k=5$ in
\[
	t'=
	\tableau[sY]{1&2&3&7&11&13\\
	\bl& 4&5&8&12&14\\
	\bl&\bl&6&9&15&18\\
	\bl&\bl&\bl&10&16&19\\
	\bl&\bl&\bl&\bl&17&20\\
	\bl&\bl&\bl&\bl&\bl&21}\;.
\]
\end{example}

\begin{proof}[Proof of Theorem~\ref{thm:braid_moves_young_tableaux}]
Since by Lemma~\ref{lemma.phi bijection} $\varphi$ is a bijection, we have that
\[
	\# \{(k,t) \mid \text{$k$ a braid hook in $t\in \rSYT(\lambda)$}\} = \#\rSYT(\lambda)\;.
\]
This implies immediately that the expected number of braid hooks (which is the quotient of the two
numbers) is one.
\end{proof}

We now study how the two partial (inverse) promotion operators $\partial_k$ and $\partial_k^*$, that are
used in the bijection $\varphi$, interact. This enables us to deduce a variant of
Theorem~\ref{thm:braid_moves_young_tableaux} as a statement on full
promotion paths in right-justified tableaux.
Namely, let $t' \in \rSYT(\lambda)$, $k$ a braid hook in $t'$, and $t=\varphi(k,t')$.
When starting at a braid hook $k$, the operators $\partial_k$ and $\partial^*_k$ commute:
\begin{equation}
  \label{eq:promotion_diagram}
  \begin{tikzpicture}
    \matrix (m)[matrix of math nodes,row sep=3em,column sep=6em,minimum width=2em,text height=1.5ex,text depth=0.25ex]
     {
      t'   & t_r         \\
      t_l & t=\varphi(k,t')   \\
    };
    \path[->] (m-1-1) edge node [above] {$\partial_k$}   (m-1-2);
    \path[->] (m-2-1) edge node [above] {$\partial_k$}   (m-2-2);
    \path[->] (m-1-1) edge node [left ] {$\partial^*_k$} (m-2-1);
    \path[->] (m-1-2) edge node [right ] {$\partial^*_k$} (m-2-2);
  \end{tikzpicture}
\end{equation}
The nice feature of this diagram is that $t_r$ is obtained from $t_l$
by applying a full promotion operator: $t_r=t_l.\partial$.
Hence, on the $t_l$ side, we can focus on the
combinatorics of just the usual promotion path.

\begin{example}
\label{example.sliding paths2}
Continuing Example~\ref{example.sliding paths1} we obtain the commutative diagram:
\newsavebox{\shiftedtableautp}
\sbox{\shiftedtableautp}{
	$t'=$
	\tableau[sY]{1&2&3&7&11&13\\
	\bl& 4&5&8&12&14\\
	\bl&\bl&6&9&15&18\\
	\bl&\bl&\bl&10&16&19\\
	\bl&\bl&\bl&\bl&17&20\\
	\bl&\bl&\bl&\bl&\bl&21}
}
\newsavebox{\shiftedtableautl}
\sbox{\shiftedtableautl}{
$t_l=$
\tableau[sY]{
  \tf 1 & \tf 2 & 4 & 7 & 11 & 13\\
  \bl & \tf 3 & \tf 5 & 8 & 12 & 14\\
  \bl & \bl & \tf 6 & \tf 9 & 15 & 18\\
  \bl & \bl & \bl & \tf 10 & \tf 16 & 19\\
  \bl & \bl & \bl & \bl & \tf 17 & \tf 20\\
  \bl & \bl & \bl & \bl & \bl & \tf 21}
}
\newsavebox{\shiftedtableautr}
\sbox{\shiftedtableautr}{
$t_r=$
\tableau[sY]{
  \lb\overlay 1 & \lb\overlay 2 & 3 & 6 & 10 & 12\\
  \bl & \lb\overlay 4 & \lb\overlay 5 & 7 & 11 & 13\\
  \bl & \bl & \lb\overlay 8 & \lb\overlay 9 & 14 & 17\\
  \bl & \bl & \bl & \lb\overlay 15 & \lb\overlay 16 & 18\\
  \bl & \bl & \bl & \bl & \lb\overlay 19 & \lb\overlay 20\\
  \bl & \bl & \bl & \bl & \bl & \lb\overlay 21}
}
\newsavebox{\shiftedtableaut}
\sbox{\shiftedtableaut}{
	$t=$
	\tableau[sY]{\tf \lb \overlay 1&\tf \lb \overlay 2&\lb \overlay 4&6&10&12\\
	\bl& \tf 3&\tf \lb \overlay5&\tf 7&11&13\\
	\bl&\bl&\lb \overlay 8&\tf \lb \overlay 9& \tf 14& 17\\
	\bl&\bl&\bl&\lb \overlay 15&\tf \lb \overlay 16&\tf 18\\
	\bl&\bl&\bl&\bl&\lb \overlay 19&\tf \lb \overlay 20\\
	\bl&\bl&\bl&\bl&\bl&\tf \lb \overlay 21}
}
\begin{displaymath}
  \begin{tikzpicture}
    \matrix (m) [matrix of math nodes,row sep=3em,column sep=5em,minimum width=2em] {
      \usebox{\shiftedtableautp}
      &
      \usebox{\shiftedtableautr}
      \\
      \usebox{\shiftedtableautl}
      &
      \usebox{\shiftedtableaut}
  \\
    };
    \path[->] (m-1-1) edge node [above] {$\partial_5$}   (m-1-2);
    \path[->] (m-2-1) edge node [above] {$\partial_5$}   (m-2-2);
    \path[->] (m-1-1) edge node [left ] {$\partial^*_5$} (m-2-1);
    \path[->] (m-1-2) edge node [right ] {$\partial^*_5$} (m-2-2);
  \end{tikzpicture}
\end{displaymath}
Note that the promotion path of $t_l$ is made of the first half of the
promotion path of $t$ and the second half of the inverse promotion
path of $t$. Note also that, viewing the promotion path of $t_l$ as a
Dyck path, it has a peak of height one with corresponding values in the
tableau of the form $(*,k,k+1)$ (here $k=5$).
\end{example}

The tableau in Example~\ref{example.sliding paths2} was of shifted staircase shape,
so that the promotion and inverse promotion paths could easily be viewed as Dyck paths.
For a general right-justified tableau $t\in\rSYT(\lambda)$, we define the analogous notion
of a \defn{left partial braid hook} to be an inner corner with values $(*,k,k+1)$ of $t_l$ 
that lies on the promotion path. The symmetric situation appears in $t_r$ and we define a 
\defn{right partial hook} to be an inner corner with values $(k-1,k,*)$ of $t_r$ 
that lies on the inverse promotion path. We thus obtain the following corollary to
Theorem~\ref{thm:braid_moves_young_tableaux}.

\begin{corollary}
  \label{corollary.partial hooks}
  The commutative diagram~\eqref{eq:promotion_diagram} gives bijections between:
  \begin{enumerate}
  \item Pairs $(k,t')$ where $t' \in \rSYT(\lambda)$ and $k$ is a braid hook of $t'$.
  \item Pairs $(k,t_l)$ where $t_l \in \rSYT(\lambda)$ and $k$ is a left partial braid hook of $t_l$.
  \item Pairs $(k,t_r)$ where $t_r \in \rSYT(\lambda)$ and $k$ is a right partial braid hook of $t_r$.
  \item Right-justified tableaux $t \in \rSYT(\lambda)$.
  \end{enumerate}
  In particular, the number of left partial braid hooks in all tableaux in $\rSYT(\lambda)$ has expected
  value one.
\end{corollary}

\section{Half-right-justified tableaux}
\label{sec:half_right_justified}

In this section, we turn our attention to shifted SYT of \defn{half-right-justified shape}. An SYT $t$ is half-right-justified 
of shape $\lambda = (\lambda_1,\lambda_2,\ldots,\lambda_\ell)$ if $\lambda_1>\lambda_2 > \cdots > \lambda_\ell$ 
are strictly decreasing and $t$ is justified so that the rightmost cell of each row is one step below and to the left of the 
rightmost cell of the previous row. We denote the set of half-right-justified SYT of shape $\lambda$
by $\hrSYT(\lambda)$.

This definition is motivated by the fact that tableaux of these shapes can be adjoined to their reflection to create tableaux of right-justified shapes, to which the results of Section \ref{section.right-justified} apply. See Figure \ref{figure.half right} for an example. Braid hooks in half-right-justified tableaux are still defined as in Definition~\ref{definition.braid hook} (with $\ShSYT(\Delta_n)$ replaced by $\hrSYT(\lambda)$).

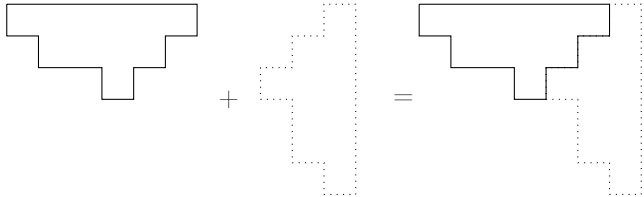
\begin{figure}
\begin{tikzpicture}[scale=\squaresize / 1cm * 1pt]
\draw (0,3) -- (6,3) \stepd \stepl \stepd \stepl \stepd \stepl \stepu \stepl \stepl \stepu \stepl \stepu;
\draw (7,0) node {+}; 
\begin{scope}[xscale=-1, rotate = 90, shift={(-3,8)}, dotted]
\draw (0,3) -- (6,3) \stepd \stepl \stepd \stepl \stepd \stepl \stepu \stepl \stepl \stepu \stepl \stepu;
\end{scope}
\draw (12.5,0) node {=};
\begin{scope}[shift={(13,0)}]
\draw (0,3) -- (6,3) \stepd \stepl \stepd \stepl \stepd \stepl \stepu \stepl \stepl \stepu \stepl \stepu;
\begin{scope}[xscale=-1, rotate = 90, shift={(-3,4)}, dotted]
\draw (0,3) -- (6,3) \stepd \stepl \stepd \stepl \stepd \stepl \stepu \stepl \stepl \stepu \stepl \stepu;
\end{scope}
\end{scope}
\end{tikzpicture}
\caption{A half-right-justified shape is joined to its reflection to create a right-justified shape.}
\label{figure.half right}
\end{figure}

As a specific example, it is natural to look at half-right-justified tableaux of trapezoidal shape, which coincide with shifted tableaux of trapezoidal shape.  These are SYT of shape $\Delta_n^t = (2n+1, 2n-1, \dots, 5,3,1)$, justified so that the center cells of each row are in the same column. Example~\ref{example.trapezoidal} illustrates such an SYT.

By a theorem of M.~Haiman, SYT of shifted trapezoidal shape are in bijection with the set of all reduced words for the longest element in type $B$~\cite{haiman.1992}, although we no longer have the interpretation as braid moves on these words.  Regardless, by the heap construction of Section~\ref{section.heaps}, it is clear that the reduced words in the commutation class of $\prod_{i=1}^{2n-1}s_{i}s_{i-2}\cdots s_{2-(i\mod 2)} \in \mathfrak{S}_{2n}$ are in bijection with such tableaux (and that braid relations in the words correspond to braid hooks in the tableaux).

\begin{example}
\label{example.trapezoidal}
The following trapezoidal tableau is in $\hrSYT(\Delta_2^t)$ 
\[
	\tableau[sY]{1&2&3&4&9\\ \bl&\darkblue{5}&\darkblue{6}&8&\bl \\ \bl&\bl&\darkblue{7}&\bl&\bl \\}\;.
\]
It contains one braid hook, shown in blue. The letters $\{6,7,8\}$ do not form a braid hook --- by definition, braid hook 
configurations can only occur on the lower left boundary.
\end{example}

We prove the following theorem.
\begin{theorem}
	The expected number of braid hooks in $\hrSYT(\lambda)$ is at most one half. 
	If $\lambda_1 \geq \lambda_2 + 2$ and $\lambda_\ell = 1$, then 
	the expected number of braid hooks in $\hrSYT(\lambda)$ is exactly one half.
\label{thm:braid_moves_half_right_young_tableaux}
\end{theorem}

Note that half-right-justified tableaux of trapezoidal shape satisfy $\lambda_1=\lambda_2+2$ and
$\lambda_\ell = 1$, so Theorem~\ref{thm:braid_moves_half_right_young_tableaux} implies 
Theorem~\ref{thm:braid_moves_half_young_tableaux}.

\subsection{Proof of Theorem~\ref{thm:braid_moves_half_right_young_tableaux}: Injective Case}
\label{subsection.injective}
We use the techniques of Section~\ref{section.right-justified}. As in that section, we define a map 
\begin{equation}
	\psi \colon \{(k,t) \mid \text{$k$ a braid hook in $t\in \hrSYT(\lambda)$}\} 
	\to \hrSYT(\lambda)
\end{equation}
using the partial promotion and inverse promotion operators $\partial_k := \tau_k \tau_{k+1} \cdots \tau_{|\lambda|-1}$ 
and $\partial^*_k := \tau_{k-1} \tau_{k-2} \cdots \tau_1$ such that
\[
	\psi(k,t) := t.\partial^*_k \partial_k\;.
\]

As before, we seek to understand the image of the map $\psi$. Recall that on right-justified
tableaux of shape $\lambda$, the map $\varphi$ is a bijection, which shows that the expected number of braid 
hooks in a tableau in $\rSYT(\lambda)$ is one. To prove Theorem~\ref{thm:braid_moves_half_right_young_tableaux}, 
we will show that on $\hrSYT(\lambda)$ the map $\psi$ is an injection whose image is at most half of $\hrSYT(\lambda)$, 
and that if $\lambda_1 \ge \lambda_2+2$ and $\lambda_\ell = 1$, then the image of $\psi$ is exactly half of 
$\hrSYT(\lambda)$.

As in the proof of Lemma~\ref{lemma.phi bijection}, we consider an element $t \in \hrSYT(\lambda)$, and examine 
the promotion and inverse promotion paths $\mathcal{L}$ and $\mathcal{R}$. Each appearance of $t$ in the image 
of $\psi$ corresponds to a crossing from ($\mathcal{R}$ above $\mathcal{L}$) to ($\mathcal{L}$ above $\mathcal{R}$) 
on a left inner corner. It is impossible for the reverse crossing to occur, as configuration~\eqref{equation.forbidden1} is 
forbidden, so the paths $\mathcal{L}$ and $\mathcal{R}$ cross at most once in $t$, showing that $\psi$ is injective.

However, for $t \in \hrSYT(\lambda)$, it is no longer true that the paths $\mathcal{L}$ and $\mathcal{R}$ must cross. 
In the top left corner, the two paths overlap, and so by our convention we consider $\mathcal{R}$ to be above 
$\mathcal{L}$. The path $\mathcal{R}$ will end in the cell containing the largest letter $|\lambda|$, while the path 
$\mathcal{L}$ could end in a different lower right cell. Example~\ref{example.trapezoidal no intersection} illustrates 
this behavior.

\begin{example}
\label{example.trapezoidal no intersection}
In
\[
	\tableau[sY]{\tf \lb \overlay 1&\tf \lb \overlay 2&\lb \overlay 4&\lb \overlay 5&\lb \overlay 7&\lb \overlay 12&13 \\ 
	\bl&\tf 3&\tf 6&\tf 8&11&\lb \overlay 16&\bl\\ \bl&\bl&9&\tf 10&\tf 14&\bl&\bl\\
	\bl&\bl&\bl&15&\bl&\bl&\bl\\}
	\in \hrSYT(\Delta_3^t)
\]
the path $\mathcal{R}$ is always above the path $\mathcal{L}$.
\end{example}

A tableau $t$ appears in the image of $\psi$ if and only if the paths $\mathcal{L}$ and $\mathcal{R}$ of $t$ cross. 
Hence, to prove Theorem~\ref{thm:braid_moves_half_right_young_tableaux}, it suffices to show that $\mathcal{L}$ 
and $\mathcal{R}$ cross in at most half of the tableaux in $\hrSYT(\lambda)$, and exactly half when 
$\lambda_1 \ge \lambda_2+2$ and $\lambda_\ell = 1$. 

We will now work towards pairing elements of $\hrSYT(\lambda)$ in which the paths $\mathcal{L}$ and $\mathcal{R}$ 
cross with those in which the two paths do not cross. For general shapes $\lambda$, some tableaux in which the paths 
do not cross may remain unpaired, while if $\lambda_1\ge \lambda_2+2$ and $\lambda_\ell = 1$, every such tableau 
is paired. We will use the \defn{evacuation} and \defn{dual evacuation} maps, defined as 
\begin{equation*}
\begin{split}
	\epsilon &= (\tau_1\tau_2\cdots\tau_{|\lambda|-1})(\tau_1\tau_2\cdots\tau_{|\lambda|-2})\cdots(\tau_1\tau_2)(\tau_1)\;,\\
	\epsilon^* &= (\tau_{|\lambda|-1}\tau_{|\lambda|-2}\cdots\tau_1)(\tau_{|\lambda|-1}\tau_{|\lambda|-2}\cdots
	\tau_2)\cdots(\tau_{|\lambda|-1}\tau_{|\lambda|-2})(\tau_{|\lambda|-1})\;.
\end{split}
\end{equation*}
Here the $\tau_i$ are as defined in~\eqref{equation.tau} with $\rSYT(\lambda)$ replaced by $\hrSYT(\lambda)$.

In order to prove Theorem~\ref{thm:braid_moves_half_right_young_tableaux}, it suffices to show the following proposition.

\begin{proposition}
\label{proposition.cross pairing}
If in the tableau $t \in \hrSYT(\lambda)$, the paths $\mathcal{L}$ and $\mathcal{R}$ cross, then in the tableau 
$t.\epsilon$, the paths $\mathcal{L}$ and $\mathcal{R}$ do not cross. If $\lambda_1 \geq \lambda_2 + 2$ and 
$\lambda_\ell = 1$, then the converse is also true.
\end{proposition}

Given an element $t \in \hrSYT(\lambda)$, we define the \defn{conjugate} of $t$, denoted by $t^\dagger$, as the 
tableau obtained by reflecting $t$ in the diagonal from bottom left to top right and then reversing the order of the entries. 
The tableau thus obtained has rows and columns in increasing order because the reflection takes rows and columns 
in increasing order to columns and rows in decreasing order respectively, and then reversing the entries produces 
columns and rows in increasing order. Hence $t^\dagger$ is an SYT, although not of the same shape as $t$.
Example~\ref{example.conjugation} illustrates this operation.

\begin{example}
\label{example.conjugation}
\[
	t = \tableau[sY]{1&2&3&4&9\\ \bl&5&6&8&\bl \\ \bl&\bl&7&\bl&\bl\\} \longrightarrow 
	\tableau[sY]{\bl&\bl&9\\ \bl&8&4 \\ 7&6&3\\ \bl&5&2\\ \bl&\bl&1\\} \longrightarrow 
	\tableau[sY]{\bl&\bl&1\\ \bl&2&6 \\ 3&4&7\\ \bl&5&8\\ \bl&\bl&9\\} = t^\dagger \;.
\]
\end{example}

We will need the following relations between promotion, evacuation and conjugation.

\begin{lemma}
\label{lemma.identities}
The operators $\partial, \partial^*, \epsilon, \epsilon^*$ and $\dagger$ obey the following relations:
\begin{align*}
\partial^* & = \partial^{-1} \\
\dagger^2 & = 1 \\
\dagger\partial\dagger & = \partial^* \\
\dagger\epsilon\dagger & = \epsilon^* \\
\epsilon^2 & = (\epsilon^*)^2 = 1 \\
\epsilon\partial & = \partial^*\epsilon \\
\epsilon^*\partial & = \partial^*\epsilon^* 
\end{align*}
\end{lemma}
\begin{proof}
That $\partial^* = \partial^{-1}$ is immediate from their definitions in terms of the involutions $\tau_i$. The conjugation 
map $\dagger$ is self-inverse because both reflecting the tableau and reversing the entries are self-inverse, and 
commute with one another. The map $\dagger$ reverses labels but otherwise preserves the poset structure, so 
we have that $\dagger\tau_i\dagger = \tau_{|\lambda|-i}$. Hence $\dagger\partial\dagger = \partial^*$ and 
$\dagger\epsilon\dagger = \epsilon^*$.

It is a result of Sch\"{u}tzenberger that $\epsilon^2 = 1$, see for example~\cite[Theorem 2.1]{stanley.2009}. 
The dual evacuation operator $\epsilon^*$ is the conjugate of $\epsilon$ by $\dagger$, so it is also an involution. 
That $\epsilon\partial = \partial^*\epsilon$ is also stated in~\cite[Theorem 2.1]{stanley.2009}. The dual statement, 
that $\epsilon^*\partial = \partial^*\epsilon^*$, may be obtained by conjugating the previous identity by $\dagger$.  
\end{proof}

Given $t \in \hrSYT(\lambda)$, let us define the \defn{staircase pair} $(t, (t.\epsilon)^\dagger)$ as follows. Take 
$t$ and $t^\dagger$, and add $|\lambda|$ to each entry in $t^\dagger$. As in Figure \ref{figure.half right}, align 
the two tableaux so that the top cell of $t^\dagger$ is to the right of the rightmost cell of $t$, and consider the union 
of these two tableaux as a larger tableau. Because $t$ and $t^\dagger$ are SYT, the staircase pair 
$(t, (t.\epsilon)^\dagger)$ is an SYT. This construction is illustrated in the next example.

\begin{example}
\label{example.staircase pair}
With $t$ as in Example~\ref{example.conjugation}, we have
\[
	t.\epsilon = \tableau[sY]{1&2&3&4&6\\ \bl&5&7&9&\bl \\ \bl&\bl&8&\bl&\bl\\}, 
	\quad (t.\epsilon)^\dagger = \tableau[sY]{\bl&\bl&4\\ \bl&1&6 \\ 2&3&7\\ \bl&5&8\\ \bl&\bl&9\\}\;, 
	\quad (t.\epsilon)^\dagger+9 = \tableau[sY]{\bl&\bl&13\\ \bl&10&15 \\ 11&12&16\\ \bl&14&17\\ \bl&\bl&18\\}\;,
\]
so that
\[
	(t, (t.\epsilon)^\dagger) = \tableau[sY]{1&2&3&4&9&13\\ \bl&5&6&8&10&15 \\ \bl&\bl&7&11&12&16\\ 
	\bl&\bl&\bl&\bl&14&17\\ \bl&\bl&\bl&\bl&\bl&18\\}\;.
\]
\end{example}

\begin{remark}
We could also have defined staircase pairs using dual evacuation, because $(t.\epsilon)^\dagger = t^\dagger.\epsilon^*$.
\end{remark}

Note that a staircase pair $(t, (t.\epsilon)^\dagger)$ is an SYT of right-justified shape, so that by the results of 
Section~\ref{section.right-justified} the paths $\mathcal{L}$ and $\mathcal{R}$ cross exactly once. For $t$ of general 
shape $\lambda$, this crossing might take place within the subtableau $t$, within the subtableau $(t.\epsilon)^\dagger$, 
or overlapping each of the two. For $t \in \hrSYT(\Delta_n^t)$, though, the crossing must be either entirely within the 
subtableau $t$, or entirely within the subtableau $(t.\epsilon)^\dagger$, because there are no braid hooks crossing 
the boundary between the two subtableaux.

Let $t \in \hrSYT(\lambda)$. We will now examine the relation between the paths $\L$ and $\R$ of a staircase pair 
$(t, (t.\epsilon)^\dagger)$ and the paths $\L$ and $\R$ of the subtableaux $t$ and $(t.\epsilon)^\dagger$. 
Let the promotion and inverse promotion paths of the staircase pair $(t, (t.\epsilon)^\dagger)$ be denoted by 
$\mathcal{L}_s$ and $\mathcal{R}_s$, while the promotion and inverse promotion paths in the subtableau $t$ are 
denoted by $\mathcal{L}_1$ and $\mathcal{R}_1$ and the promotion and inverse promotion paths in the subtableau 
$(t.\epsilon)^\dagger$ are denoted by $\mathcal{L}_2$ and $\mathcal{R}_2$.

We prove Proposition~\ref{proposition.cross pairing} via the following sequence of lemmas.

\begin{lemma}
The restriction of $\mathcal{L}_s$ to the subtableau $t$ is exactly $\mathcal{L}_1$. Likewise, the restriction of 
$\mathcal{R}_s$ to the subtableau $(t.\epsilon)^\dagger$ is exactly $\mathcal{R}_2$.
\end{lemma}
\begin{proof}
Both of the paths $\mathcal{L}_s$ and $\L_1$ can be constructed by starting at the cell containing 1 and 
continually moving down or to the right, to whichever cell has the smaller entry. Because every entry in the 
subtableau $t$ is smaller than every other entry of the staircase pair tableau, these two paths will overlap 
until $\L_s$ leaves the subtableau $t$, at which point $\L_1$ terminates. Hence the restriction of 
$\mathcal{L}_s$ to $t$ is exactly $\mathcal{L}_1$.

Similarly, the paths $\mathcal{R}_s$ and $\R_2$ are both obtained by starting at the cell containing $2|\lambda|$ and repeatedly moving up or to the left, to whichever cell has the larger entry. These two paths will overlap until $\R_s$ leaves the subtableau $(t.\epsilon)^\dagger$, at which point $\R_2$ terminates. Therefore the restriction of $\mathcal{R}_s$ to $(t.\epsilon)^\dagger$ is exactly 
$\mathcal{R}_2$.
\end{proof}

We now state a lemma of~\cite{pon.wang.2011}, and deduce a very similar dual statement in Lemma~\ref{lemma.pw2}. 
Note that with respect to~\cite{pon.wang.2011}, we have interchanged the definitions of promotion and inverse promotion, 
and those of evacuation and dual evacuation, following~\cite{stanley.2009} rather than~\cite{edelmann.greene.1987}.

\begin{lemma}{\cite[Lemma~3.4]{pon.wang.2011}}
\label{lemma.pw1}
If the letter $|\lambda|$ is in cell $(i,j)$ of $t$, then the promotion path $\L$ of $t.\epsilon$ ends on cell $(i,j)$ of 
$t.\epsilon$. 
\end{lemma}
\begin{proof}
From Lemma~\ref{lemma.identities}, we know that $t.(\epsilon\partial) = t.(\partial^*\epsilon)$. Working first with 
the right hand side, we see that 
\[
\partial^*\epsilon = (\tau_1\tau_2\cdots\tau_{|\lambda|-2})(\tau_1\tau_2\cdots\tau_{|\lambda|-3})\cdots(\tau_1\tau_2)(\tau_1)\;.
\]
Note that the operator $\partial^*\epsilon$ does not move the letter $|\lambda|$, as $\tau_{|\lambda|-1}$ does not appear. 
Therefore the position of $|\lambda|$ is the same in $t.(\partial^*\epsilon)$ as in $t$. But 
$t.(\epsilon\partial) = t.(\partial^*\epsilon)$, so the position of $|\lambda|$ must be the same in $t.(\epsilon\partial)$ as in $t$.

The position of $|\lambda|$ in $t.(\epsilon\partial)$ is the lower right endpoint of the path $\L$ in $t.\epsilon$, by the 
sliding definition of promotion. This completes the proof.
\end{proof}

\begin{lemma}
\label{lemma.pw2}
If the letter $1$ is in cell $(i,j)$ of $(t.\epsilon)^\dagger$, then the inverse promotion path $\R$ of 
$(t.\epsilon)^\dagger.\epsilon^*$ ends on cell $(i,j)$ of $(t.\epsilon)^\dagger.\epsilon^*$. 
\end{lemma}
\begin{proof}
Note that the conjugation map $\dagger$ reverses the labels and interchanges the notions of (below or to the left) and (above or to the right), so it takes the path $\L$ of $t.\epsilon$ to the path $\R$ of $(t.\epsilon)^\dagger$. 

Applying Lemma~\ref{lemma.pw1} to the tableau $(t.\epsilon)$ gives that if the letter $|\lambda|$ is in cell $(i,j)$ of $(t.\epsilon)$, then the promotion path $\L$ of $t$ ends on cell $(i,j)$ of $t$, because $\epsilon^2 = 1$.

Application of the map $\dagger$ to this statement completes the proof, using from Lemma~\ref{lemma.identities} that $(t.\epsilon)^\dagger.\epsilon^* = t^\dagger$. 
\end{proof}

\begin{lemma}
\label{lemma.mid overlap}
The path $\R_s$ passes through the cell containing $|\lambda|$, the maximal entry in $t$. Likewise, the path 
$\L_s$ passes through the cell containing $|\lambda| + 1$, the minimal entry in $(t.\epsilon)^\dagger$.
\end{lemma}
\begin{proof}
Let the letter $|\lambda|$ be in cell $(i,j)$ of $t$. From Lemma~\ref{lemma.pw1}, we have that the promotion path 
$\L$ of $t.\epsilon$ ends on cell $(i,j)$ of $t.\epsilon$. As noted in the proof of Lemma~\ref{lemma.pw2}, the conjugation map $\dagger$ takes the path $\L$ of $t.\epsilon$ to the path $\R$ of $(t.\epsilon)^\dagger$. 

Hence, the path $\R_2$ of $(t.\epsilon)^\dagger$ passes through the image under $\dagger$ of the cell $(i,j)$ in 
$(t.\epsilon)^\dagger$. As $\R_s$ agrees with $\R_2$ up to this point, we have that $\R_s$ passes through the cell 
$(i,j)^\dagger$. By the construction of the staircase pair, this cell is immediately to the right of the cell $(i,j)$, 
which is in $t$. The path $\R_s$ moves from the cell $(i,j)^\dagger$ to the cell above or to the left, whichever 
has the larger entry. But both of these cells (if they exist), are in the subtableau $t$, and $|\lambda|$ is the 
largest entry in $t$. Therefore $\R_s$ passes through $(i,j)$, the cell containing $|\lambda|$.

The second part of the lemma is given by applying this result to the `transposed' staircase pair $(t.\epsilon,t^\dagger)$, noting that this operation swaps the paths $\L_s,\L_1$ and $\L_2$ with $\R_s,\R_2 $ and $\R_1$ respectively.
\end{proof}

\begin{lemma}
The restriction of $\R_s$ to $t$ is exactly $\R_1$, and the restriction of $\L_s$ to $(t.\epsilon)^\dagger$ is 
exactly $\L_2$.
\end{lemma}
\begin{proof}
The path $\R_1$ starts at $|\lambda|$ and moves up or to the left, to whichever cell has the larger entry. 
The path $\R_s$ moves in the same way, and passes through the cell containing $|\lambda|$ by 
Lemma~\ref{lemma.mid overlap}. The subtableau $t$ has no cells below or to the right of the cell 
containing $|\lambda|$, so the restriction of $\R_s$ to $t$ is exactly $\R_1$.

Similarly, the path $\L_2$ starts at $|\lambda|+1$ and moves down or to the right, to whichever cell has 
the smaller entry. The path $\L_s$ moves in the same way, and passes through the cell containing 
$|\lambda|+1$ by Lemma~\ref{lemma.mid overlap}. The subtableau $(t.\epsilon)^\dagger$ has no cells 
above or to the left of the cell containing $|\lambda|+1$, so the restriction of $\L_s$ to $(t.\epsilon)^\dagger$ 
is exactly $\L_2$.
\end{proof}

\begin{corollary}
Given $t \in \hrSYT(\lambda)$, the path $\mathcal{L}$ of the staircase pair $(t, (t.\epsilon)^\dagger)$ is 
the concatenation of the paths $\mathcal{L}$ in the subtableaux $t$ and $(t.\epsilon)^\dagger$. 
The same is true when each $\L$ is replaced by an $\R$. 
\end{corollary}

\begin{corollary}
\label{corollary.crossings}
Given $t \in \hrSYT(\lambda)$, the paths $\mathcal{L}$ and $\mathcal{R}$ of $t$ cross if and only if in 
the staircase pair $(t, (t.\epsilon)^\dagger)$, the paths $\mathcal{L}$ and $\mathcal{R}$ of $(t, (t.\epsilon)^\dagger)$ 
cross in the subtableau $t$. 

Likewise, the paths $\mathcal{L}$ and $\mathcal{R}$ of the staircase pair $(t, (t.\epsilon)^\dagger)$ cross in the 
subtableau $(t.\epsilon)^\dagger$ if and only if the paths $\mathcal{L}$ and $\mathcal{R}$ of $(t.\epsilon)^\dagger$ cross.
\end{corollary}

\begin{lemma}
For any $t \in \hrSYT(\lambda)$, if the paths $\L$ and $\R$ cross in $t$ then in $t.\epsilon$, the paths $\L$ and $\R$ 
do not cross. 

If $\lambda_1 \geq \lambda_2 + 2$ and $\lambda_\ell = 1$, then the converse is also true. That is, exactly one of 
$t$ and $t.\epsilon$ has its paths $\L$ and $\R$ cross.
\end{lemma}
\begin{proof}
If the paths $\L$ and $\R$ did cross in $t.\epsilon$, then in $(t.\epsilon)^\dagger$ the paths $\R$ and $\L$ would 
cross, as the map $\dagger$ takes the paths $\L$ and $\R$ in $t.\epsilon$ to the paths $\R$ and $\L$ in 
$(t.\epsilon)^\dagger$.

But then by Corollary~\ref{corollary.crossings}, in the staircase pair $(t, (t.\epsilon)^\dagger)$ the paths $\L$ and 
$\R$ would cross at least twice, which contradicts the fact that in a right-justified tableau, $\L$ and $\R$ may cross 
at most once. This completes the proof of the first part of the lemma.

If $\lambda_\ell = 1$, then every braid hook in the staircase pair $(t, (t.\epsilon)^\dagger)$ is entirely contained 
within one of the subtableaux $t$ and $(t.\epsilon)^\dagger$. This is because a braid hook spanning both 
subtableaux would be formed only of cells in the bottom row of $t$ and in the left column of $(t.\epsilon)^\dagger$. 
If $\lambda_\ell = 1$, then there are only two such cells. 

If $\lambda_1 \geq \lambda_2 + 2$ then the staircase pair $(t, (t.\epsilon)^\dagger)$ is a right-justified tableau 
whose first row is longer that its second, and with a last row of a single cell. As in the proof of 
Lemma~\ref{lemma.phi bijection}, these are the conditions under which we know that the paths $\L$ and $\R$ 
of $(t, (t.\epsilon)^\dagger)$ cross exactly once. Because they must cross on a braid hook, this crossing must 
happen entirely within one of the subtableaux $t$ and $(t.\epsilon)^\dagger$. By Corollary~\ref{corollary.crossings}, 
in one of the tableaux $t$ and $(t.\epsilon)^\dagger$, the paths $\L$ and $\R$ cross.

Finally, the paths $\L$ and $\R$ of $(t.\epsilon)^\dagger$ cross if and only if the paths $\L$ and $\R$ of 
$t.\epsilon$ cross, completing the proof.
\end{proof}

We have shown that the paths $\mathcal{L}$ and $\mathcal{R}$ cannot cross in both $t$ and $t.\epsilon$, and 
that if $\lambda_1 \geq \lambda_2 + 2$ and $\lambda_\ell = 1$, then they cross in exactly one of those tableaux. 
This completes the proof of Proposition~\ref{proposition.cross pairing} and thus of 
Theorem~\ref{thm:braid_moves_half_right_young_tableaux}.

\subsection{Surjective Case}

The map $\varphi$ of Section~\ref{section.right-justified} is a bijection, because for the tableau shapes under 
consideration in that section, the paths $\L$ and $\R$ always cross exactly once. In 
Section~\ref{subsection.injective}, the map $\psi$ is an injection, because for the relevant shapes, the paths 
$\L$ and $\R$ may cross either 0 or 1 times. Understanding the image of $\psi$ allows us to determine the 
expected number of braid hooks for a tableau of, for example, trapezoidal shape.

In this section, we consider tableaux of \defn{skew right-justified shape}.
Let $\mu \subset \lambda=(\lambda_1,\ldots,\lambda_\ell)$ be two partitions.
Then we may consider standard tableaux of skew right-justified shape $\lambda/\mu$,
denoted by $\rSYT(\lambda/\mu)$.
If the skew shape $\lambda/\mu$ is connected (i.e., for each pair of consecutive rows, there are at least two cells 
(one in each row) which have a common edge), $\lambda_1>\lambda_2$ and $\lambda_\ell=1$, 
then the paths $\L$ and $\R$ in a tableau $t \in \rSYT(\lambda/\mu)$
must cross at least once and potentially cross more than once. In this case, the corresponding map $\psi$ is surjective. 

An example of a connected skew right-justified shape and a skew right-justified tableau
with paths $\L$ and $\R$ that cross more than once is given in Figure~\ref{figure.top.diagonal}. 
In general, the path $\R$ is above $\L$ in the top left corner if $\lambda_1>\lambda_2$. In the bottom right, 
$\L$ is above $\R$ if $\lambda_\ell=1$, so the paths cross at least once. Unlike the shapes we have previously 
considered, it is possible for the second hook of~\eqref{equation.hooks} to appear, in the top right corner. If this 
happens, then the paths cross in the other direction --- from ($\L$ above $\R$) to ($\R$ above $\L$). 

\begin{figure}
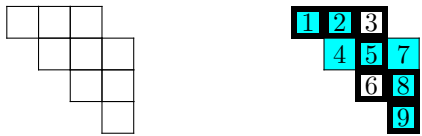

\[
	\tableau[sY]{&&&\bl\\ \bl&&&\\ \bl&\bl&& \\ \bl&\bl&\bl&\\} \qquad \qquad \qquad
	\tableau[sY]{\tf \lb \overlay 1& \tf \lb \overlay 2& \tf 3&\bl\\ \bl&\lb \overlay 4&\tf \lb \overlay 5
	& \lb \overlay 7\\ \bl&\bl&\tf 6 & \tf \lb \overlay 8 \\ \bl&\bl&\bl& \tf \lb \overlay 9\\}
\]
\caption{Left: a connected skew right-justified shape $\lambda/\mu=(4,3,2,1)/(1)$.
Right: a tableau in $\rSYT(\lambda/\mu)$ in which the paths $\L$ and $\R$ cross more than once due to the 
jagged top right boundary.}
\label{figure.top.diagonal}
\end{figure}

While it is possible for there to be more than one crossing, the difference between the number of crossings of each 
type must be exactly one. That is, there is exactly one more crossing on the lower left boundary than on the upper 
right boundary. The precise statement is given in the following proposition.

\begin{proposition}
Let $\mu\subset \lambda=(\lambda_1,\ldots,\lambda_\ell)$ be two partitions such that $\lambda/\mu$ is connected,
$\lambda_1>\lambda_2$, and $\lambda_\ell=1$. Then the promotion and inverse promotion paths
$\L$ and $\R$ in a tableau $t\in \rSYT(\lambda/\mu)$ cross at least once, and the difference between
the number of crossings ($\R$ above $\L$) to ($\L$ above $\R$) minus the number of the opposite crossings, is one.
\end{proposition}

Translating this back via X.~Viennot's heap map $\nu^{-1}$ to commutation classes in $\mathfrak{S}_n$, states that in
commutation classes corresponding to connected skew shapes the expected difference between the number of 
braid moves of the form $s_is_{i+1}s_i$ and the number of braid moves of the form $s_{i+1}s_is_{i+1}$ is one.
Note that, since the shapes are skew, the words in $\mathfrak{S}_n$ are not necessarily reduced.

\begin{example}
The statement corresponding to the shape of Figure \ref{figure.top.diagonal} is that the expected difference between `up' 
and `down' braid moves in the commutation class of the word $\mathbf{w} := (s_1s_2s_3)(s_1s_2s_3)(s_1s_2)(s_1)$ 
is one. Note that $\mathbf{w}$ is not reduced. We may verify this by listing the four words in this commutation class, 
{\color{red}121}321321, {\color{red}121}{\color{blue}323}{\color{red}121}, 123{\color{red}121}321 and 
123123{\color{red}121}. Observe that there are five `up' braid moves, colored red, and one `down' braid move, 
colored blue.
\end{example}

\section{Homomesy}
\label{sec:homomesy}

In Section~\ref{sec:even_odd_homomesy}, we prove a refinement of Theorem~\ref{thm:braid_moves_young_tableaux} 
by showing that the number of braid hooks is homomesic with respect to the action of the dihedral group
$\langle \tau_o, \tau_e \rangle$, where  $\tau_o=\prod_{i \text{ odd}}\tau_i$ and $\tau_e=\prod_{i \text{ even}}\tau_i$ 
are the odd and even promotion operators, respectively. We reformulate the result in terms of reduced words
in Theorem~\ref{theorem.homomesy words} and provide a bijective proof in this setting. 
In Section~\ref{subsection.homomesy poset}
we prove an analogous result for more general posets, where the statistic of descents is proven to be homomesic
with respect to an even-odd action.

\subsection{Homomesy with respect to even-odd--promotion}
\label{sec:even_odd_homomesy}

Consider the group $G$ generated by the \defn{odd} and \defn{even promotion operators} 
$\tau_o=\prod_{i \text{ odd}}\tau_i$ and $\tau_e=\prod_{i \text{ even}}\tau_i$, respectively. 
Note that, within each operator, the $\tau_i$'s commute. Hence their relative order is
not relevant, which implies that $\tau_o$ and $\tau_e$ are
involutions. In particular, $G$ is a dihedral group.
In this section, we prove the following generalization of 
Theorem~\ref{thm:braid_hooks_homomesy_even_odd}, which corresponds to the special
case $\lambda=\Delta_n$.

\begin{theorem}
\label{theorem.homomesy}
   The number of braid hooks is homomesic with respect to the action of the dihedral
   group $\langle \tau_o,\tau_e\rangle$ on $\rSYT(\lambda)$ if and only if $\lambda_1>\lambda_2$ 
   and $\lambda_\ell=1$ for a partition $\lambda$ with $\ell$ parts.
\end{theorem}

We reformulate this result in terms of reduced words.
Define $\rW(\lambda)$ to be the commutation class of reduced words which under Viennot's 
bijection correspond to $\rSYT(\lambda)$:
\[
	\nu \colon \rW(\lambda) \to \rSYT(\lambda)\;.
\]
For $\lambda=\Delta_n$ we recover the commutation class of the reduced word
$\w_0$ for $w_0$, that is, $\rW(\Delta_n)=\Red(\w_0)$.

\begin{theorem}
  \label{theorem.homomesy words}
  The number of braid moves in $\rW(\lambda)$ has expected value at
  most one. Furthermore, the expected number of braid moves is one if
  and only if $\lambda$ satisfies
  \begin{equation}
    \label{eq.lambda_condition}
    \lambda_1>\lambda_2 \qquad \text{and} \qquad \lambda_\ell=1
  \end{equation}
  for a partition $\lambda$ with $\ell$ parts
  or, equivalently, if every word $\w \in \rW(\lambda)$ satisfies
  $\w_1\leq \w_3$ and $\w_{N-2}\geq \w_N$ where $N=|\lambda|$.
  In this case, the number of braid moves is homomesic with respect to
  $\langle \tau_o, \tau_e \rangle$-orbits.
\end{theorem}

Note that the analogous statement fails for $n=7$ if one replaces the group
$\langle \tau_o, \tau_e \rangle$ by the cyclic group generated by the gyration 
operator $\tau_o\tau_e$ or any order two subgroup. Hence,
this theorem provides an example of a homomesy under a dihedral group
action, which in general is not homomesic under the cyclic subgroup
generated by $\tau_o \tau_e$ or order two subgroups (and hence any
abelian subgroup by~\cite[Lemma 1]{roby.2015}).
We also note that the bijection $\varphi$ of \eqref{equation.varphi}
does not preserve $\langle \tau_o,\tau_e\rangle$-orbits, so one cannot use it to prove
Theorem~\ref{theorem.homomesy words} (or equivalently, Theorem~\ref{theorem.homomesy}). 

To prove Theorem~\ref{theorem.homomesy words}, we define a $\langle \tau_e, \tau_o \rangle$-orbit preserving map
\begin{equation}
	\Phi \colon \{ (k,\w) \mid \w \in \rW(\lambda),
        \text{$k$ a braid in $\w$} \}
	\to \rW(\lambda)
\end{equation}
by
\begin{equation}
  \Phi(k,\w) := \w. \tau_{o(k-2)}\cdots \tau_{o(1)}\;,
\end{equation}
where for convenience:
\begin{equation}
\label{equation.tau oi}
  \tau_{o(i)} :=
  \begin{cases}
    \tau_o & \text{if $i$ is odd},\\
    \tau_e & \text{if $i$ is even.}
  \end{cases}
\end{equation}

Theorem~\ref{theorem.homomesy words} is then a direct consequence of the
following lemma.
\begin{lemma}
  \label{lemma.gyration}
  $\Phi$ is injective. Furthermore $\Phi$ is a bijection if and only
  if $\lambda$ satisfies Equation~\eqref{eq.lambda_condition}.
\end{lemma}

To prove Lemma~\ref{lemma.gyration}, we need some preliminary
notation and results. For simplicity, we write all reduced words
$s_{i_1} \cdots s_{i_k}$ simply as a word $i_1 \ldots i_k$.
Take $\w\in \rW(\lambda)$. Recall that $\w$ cannot contain a factor
of the form $aa$ (which we call the \defn{quadratic rule}) and that,
if it contains a factor of the form $aba$, then $aba=a(a+1)a$ (which
we call the \defn{braid rule}) by a slight extension of Lemma~\ref{lemma.up down braid}. 
We say that $1<k<N$ is a \defn{braid} in $\w$ if there is a braid $a(a+1)a$ with the $a+1$ 
in position $k$ of $\w$.

For $j\ge 0$, define
\begin{equation}
\label{equation.wj}
    \w^{(j)} := \w. \tau_{o(1)} \cdots \tau_{o(j)}\;.
\end{equation}
Note that $\w^{(j)}$ runs through the $\langle \tau_o, \tau_e \rangle$-orbit of $\w$. As it moves 
through the first half of the orbit, we follow what happens in a moving
window of length $2$, setting $a_i:=\w^{(i-2)}_{i-1}$ and
$c_i:=\w^{(i-2)}_{i+1}$. Here is an example for $\w=1231423121 \in \Red(\mathbf{w}_0)$:
\renewcommand{\r}[1]{{\color{red}#1}}
\begin{displaymath}
\begin{array}{|c|c|c|c|c|}
\hline
i & \w^{(i-2)} & a_i & c_i & c_i-a_i \\\hline
2 & \r12\r31423121 & 1 & 3 & 2\\
3 & 1\r21\r3241321 & 2 & 3 & 1\\
4 & 12\r13\r214321 & 1 & 2 & 1\\
5 & 123\r12\r14321 & 1 & 1 & 0\\
6 & 1231\r24\r1321 & 2 & 1 & -1\\
7 & 12134\r23\r121 & 2 & 1 & -1\\
8 & 121342\r31\r21 & 3 & 2 & -1\\
9 & 1231241\r32\r1 & 3 & 1 & -2\\\hline
\end{array}
\end{displaymath}
Note that there exists a unique position $k$ where $a_k=c_k$, namely
$k=5$; for $i<k$, $a_i<c_i$ while for $i>k$, $a_i>c_i$. In fact, $k$
is the position of a braid in $\w^{(k-2)}$. This implies that $\w$
admits exactly one preimage by $\Phi$, namely
$\Phi^{-1}(\w)=(k, \w^{(k-2)})$.

We now move on to proving that this is a general feature whenever
$\lambda$ satisfies Equation~\eqref{eq.lambda_condition}; this implies
that $\Phi$ is indeed a bijection. When the conditions are not
satisfied, uniqueness still holds but existence fails for at least one
word $\w\in \rW(\lambda)$, and surjectivity will be lost.

\begin{lemma}
\label{lemma.comparison}
Let $\w \in \rW(\lambda)$, $1<i<N-1$ and define $\w'=\w.\tau_{o(i+1)}$.
Then, $\w_{i-1}<\w_{i+1}$ if and only if $\w'_{i}\le \w'_{i+2}$.
\end{lemma}
\begin{proof}
  Let $abcd$ and $xyzt$ be the subwords of $\w$ and $\w'$ at positions
  $i-1,\ldots,i+2$. With this notation, we want to prove that
  \begin{equation}
    \label{equation.sign}
    c-a > 0 \Longleftrightarrow t-y \geq 0\;.
  \end{equation}

  From the action of $\tau_{o(i+1)}$, we have $xy=ab$ if $b=a\pm 1$ and $xy=ba$ otherwise.
  Similarly, $zt=cd$ if $d=c\pm 1$ and $zt=cd$ otherwise. It follows that $t-y$ differs from 
  $c-a$ by at most $\pm 2$. A counterexample to Equation~\eqref{equation.sign} can therefore only 
  occur if $c-a$ is close to zero, namely in one of the following three cases:
  
  \smallskip

  \noindent
  \textbf{Case 1:} $c-a=-2$ and $t-y=0$; from the action of $\tau_{o(i+1)}$, one
  necessarily has $xyzt=abcd=ab(a-2)b$ with $b=a-1$; this is forbidden
  by the braid rule.
  
  \smallskip

  \noindent
  \textbf{Case 2:} $c-a=0$; then by the braid rule $abcd=a(a+1)ad$;
  since $\w$ is reduced $d\ne a+1$; if $d=a-1$ then
  $xyzt=a(a+1)a(a-1)$ and $t-z=-2<0$; otherwise $xyzt=a(a+1)da$ and
  $t-z=-1$; in both cases Equation~\eqref{equation.sign} is satisfied.

  \smallskip

  \noindent
  \textbf{Case 3:} $c-a=\epsilon$ with $\epsilon=\pm 1$; from the action of
  $\tau_{o(i+1)}$, $xyzt$ takes one of the following forms:
  \begin{equation}
    xyzt = \begin{cases}
      a (a\pm 1)(a+\epsilon)(a+\epsilon\pm 1),\\
      a (a\pm 1)d(a+\epsilon),\\
      ba(a+\epsilon)(a+\epsilon\pm 1),\\
      bad(a+\epsilon).\\
    \end{cases}
  \end{equation}
  If the third form is $ba(a+1)a$, then $\epsilon=1$, $t-y=0$, and
  Equation~\eqref{equation.sign} is satisfied. Otherwise, using the
  quadratic and braid rules one further deduces that $y=a-\epsilon$ in
  the two first forms and that $t=a+2\epsilon$ in the third form; it
  follows that, in all forms, $t-y$ has the same sign as $\epsilon$
  and Equation~\eqref{equation.sign} is satisfied.
\end{proof}

\begin{lemma}
\label{lemma.unique}
Let $\mathbf{w}\in \rW(\lambda)$ and define $\mathbf{w}^{(j)}$ as in~\eqref{equation.wj}. Then there exists
at most one $1<k<N$ such that
\begin{equation}
\label{equation.equality}
    \mathbf{w}_{k-1}^{(k-2)} = \mathbf{w}_{k+1}^{(k-2)}.
\end{equation}
If $\lambda$ further satisfies Equation~\eqref{eq.lambda_condition},
then existence is guaranteed.
\end{lemma}

\begin{proof}
  The statement of Lemma~\ref{lemma.comparison} can be reformulated as
  $\mathbf{w}^{(i)}_{i-1}\ge \mathbf{w}^{(i)}_{i+1}$ if and only if
  $\mathbf{w}^{(i+1)}_{i}>\mathbf{w}^{(i+1)}_{i+2}$.  Hence, if
  $\mathbf{w}^{(k)}_{k-1}=\mathbf{w}^{(k)}_{k+1}$, then
  $\mathbf{w}^{(j)}_{j-1}> \mathbf{w}^{(j)}_{j+1}$ for all $j>k$.
  This implies uniqueness.

  Suppose now that $\lambda$ satisfies
  Equation~\eqref{eq.lambda_condition} and that there is no $k$ such
  that~\eqref{equation.equality} holds. Using that
  $\lambda_1>\lambda_2$ and $\lambda_\ell=1$, it follows that
  $\mathbf{w}^{(0)}_1<\mathbf{w}^{(0)}_3$ and
  $\mathbf{w}^{(N-3)}_{N-2}>\mathbf{w}^{(N-3)}_N$.  Note that
  $\mathbf{w}^{(i+1)} = \mathbf{w}^{(i)}.\tau_{o(i+1)}$, so that we
  can move from $\mathbf{w}^{(0)}$ to $\mathbf{w}^{(N-3)}$ by
  successive applications of the operator $\tau_{o(i+1)}$ for
  $1<i<N-1$.  By Lemma~\ref{lemma.comparison}, it is not possible to
  move directly from
  $\mathbf{w}_{i-1}^{(i-2)} < \mathbf{w}_{i+1}^{(i-2)}$ to
  $\mathbf{w}_{i}^{(i-1)} > \mathbf{w}_{i+2}^{(i-1)}$. This proves the
  existence of a $k$ such that~\eqref{equation.equality} holds.
\end{proof}

\begin{proof}[Proof of Lemma~\ref{lemma.gyration}]
  Recall that, by the braid rule, for any $\w\in \rW(\lambda)$ and any
  position $i$, the equality $\w_{i-1}=\w_{i+1}$ occurs if and only if
  $i$ is a braid of $\w$.

  Assume first that $\lambda$ satisfies
  Equation~\eqref{eq.lambda_condition}. Take
  $\w \in \rW(\lambda)$. By Lemma~\ref{lemma.unique}, there exists a
  unique $k$ with $1<k<N$ such that $k$ is a braid of
  $\w^{(k-2)}$. Hence $(k, \w^{(k-2)})$ is the unique preimage of $\w$
  by $\Phi$. Therefore, $\Phi$ is a bijection, as desired.

  Otherwise Lemma~\ref{lemma.unique} still guarantees that there
  exists at most one preimage of $\w$ by $\Phi$; hence $\Phi$ is still
  an injection. However, if $\lambda_1=\lambda_2$, there exists a word
  of the form $\w=120\cdots$ in $\rW(\lambda)$; for this word,
  $a_2>c_2$ and therefore $a_i>c_i$ for $2\leq i<N$; hence $k$ is
  never a braid of $\w^{(k-2)}$, and $\w$ is not in the image of
  $\Phi$. When instead $\lambda_\ell=1$ there exists some word of the
  form $\w=\cdots021$ in $\rW(\lambda)$, and the same argument
  applies. Therefore, in both cases, $\Phi$ is not surjective.
\end{proof}

\begin{remark}
It would be interesting to explain the homomesy property stated in this section
by finding an equivariant bijection from right-justified tableaux (equipped with the action of
the even and odd promotion operators) to some other combinatorial model
equipped with a natural dihedral action.
\end{remark}

\subsection{Homomesy for posets}
\label{subsection.homomesy poset}

As discussed in Section~\ref{section.heaps}, the set $\rSYT(\lambda)$ can be viewed as the set of linear extensions 
of a poset with a unique minimal and maximal element. In this section,
we provide a homomesy result of similar nature
for posets, where the statistic is descents with respect to order ideals.

Let $P$ be a finite poset with $n:=|P|$. Denote by $\mathcal{L}(P)$ the set of linear extensions of $P$
and by $J(P)$ the set of order ideals of $P$. 
For $L \in \mathcal{P}$ and $I \in \mathcal{J}(P)$, let
\[
	\des_I(L):=\{ p \in I \mid p \lessdot L^{-1}(L(p)+1) \not \in I \}
\]
be the set of elements $p$ of $I$ that are covered by an element not in $I$ whose labeling under $L$ is 
exactly one greater than the label of $p$.  We call an element $p \in \des_I(L)$ a \defn{descent} of $L$.
We can define operators $\tau_i$ for $1\le i < n$ on a linear extension $L$ by interchanging $i$ and $i+1$ in $L$
if the result is a linear extension of $P$, and $L$ otherwise. As before, 
$\tau_o=\prod_{i \text{ odd}}\tau_i$, $\tau_e=\prod_{i \text{ even}}\tau_i$, and $\tau_{o(i)}$ as
in~\eqref{equation.tau oi}.

\begin{theorem}
	Let $P$ be a poset with minimal element $\hat{0}$ and maximal element $\hat{1}$, and fix 
	$I \in \mathcal{J}(P)\setminus \{ \emptyset,P\}$.  Then there is a $\langle \tau_o,\tau_e\rangle$-orbit-preserving 
	bijection between $\{(p,L) \mid L \in \mathcal{L}(P), p \in \des_I(L)\}$ and $\mathcal{L}(P)$.  In particular, the number 
	of descents in $\mathcal{L}(P)$ is homomesic with respect to $\langle \tau_o,\tau_e\rangle$-orbits, with 
	expected value one.
\label{thm:edges}
\end{theorem}

\begin{proof}
  Given $L \in \mathcal{L}(P)$, consider the sequence of linear extensions $L_1,L_2,\ldots,L_n$ defined as
  $L_1:=L$ and $L_{i+1}:=L_i.\tau_{o(i)}.$
  As $i$ increases, the sequence of elements of $P$ labeled by $i$ in $L_i$ form a path from $\hat{0}$ to $\hat{1}$ as follows.  At each step 
  from $L_i$ to $L_{i+1}$, there are two choices:
  \begin{itemize}
    \item if $\tau_{o(i)}$ swaps the labels $i$ and $i+1$, then our path remains constant;
    \item otherwise, $i+1$ covers $i$ and so we have extended the path.
  \end{itemize}

  Since this is a path from $\hat{0}$ to $\hat{1}$, there is a unique position $k:=k(L)$ in the sequence 
  $L_1,L_2,\ldots,L_n$ such that $L_{k-1}^{-1}(k-1) \in I$ but $L_{k}^{-1}(k) \not \in I$.

  We may therefore define the $\langle \tau_o,\tau_e\rangle$-orbit-preserving bijection 
  \[
  	\Phi: \{(p,L) \mid L \in \mathcal{L}(P), p \in \des_I(L)\} \to \mathcal{L}(P)
  \]
  by
  \[
  	\Phi(p,L) := L.\tau_{o(L(p))}.\tau_{o(L(p)-1)}.\ldots.\tau_{o(1)}\;.\qedhere
  \]
\end{proof}

\begin{corollary}
	Let $P$ be a poset with $\hat{0}$ and $\hat{1}$, and fix $I \in \mathcal{J}(P)\setminus \{ \emptyset,P\}.$
	Then 
	\[
		|\mathcal{L}(P)| = \sum_{L \in \mathcal{L}(P)} \mid \des_I(L)|\;.
	\]
\label{cor:edges}
\end{corollary}

H.~Thomas has kindly provided a beautiful geometric proof of Corollary~\ref{cor:edges}.  We recall that the
\defn{order polytope} $\mathcal{O}(P)$ of $P$ is the $n$-dimensional polytope in $\mathbb{R}^P$, whose vertices are 
given by the points $\{\mathbbm{1}_I \mid I \in \mathcal{J}(P)\}.$  The volume of $\mathcal{O}(P)$ is equal to 
$|\mathcal{L}(P)|/n!$, and the facets of $\mathcal{O}(P)$ are indexed by covers $e:=p \lessdot q$ of $P$; restricting to 
a facet $F_e$, we see that its volume is given by $|\mathcal{L}(P_e)|/(n-1)!$, where $P_e$ is $P$ with the edge $e$ contracted.  In other words, the volume of the facet $F_e$ counts the number of linear extensions of $P$ such that $L(p)+1=L(q)$.
For more details, see~\cite{stanley.1986}.

\begin{proof}[Proof of Corollary~\ref{cor:edges} (H.~Thomas)]
  Let $E$ be the set of covers $\{ p \lessdot q \mid p \in I, q \not \in I\}.$  Then the order polytope $\mathcal{O}(P)$ 
  decomposes as the union of the cones with apex given by the vertex $\mathbbm{1}_I$ over the facet $F_e$, for $e \in E$:
  \[ 
  	\mathcal{O}(P) = \bigcup_{e \in E} \mathsf{Conv}(\mathbbm{1}_I,F_e)\;.
  \]
  Since the volume of the order polytope is given by the number of linear extensions, and the cones all have height 
  one, taking volumes of the decomposition above gives:
  \begin{align*} 
  \frac{|\mathcal{L}(P)|}{n!} & = \mathsf{Vol}(\mathcal{O}(P)) 
  = \sum_{e \in E} \mathsf{Vol}(\mathsf{Conv}(\mathbbm{1}_I,F_e)) \\
  & = \sum_{e \in E} \frac{1}{n} \cdot \frac{|\mathcal{L}(P_e)|}{(n-1)!} 
  = \frac{1}{n!}\sum_{p \lessdot q \in E} \sum_{L \in \mathcal{L}(P)} \mathbbm{1}_{L(p)+1=L(q)}
  = \frac{1}{n!} \sum_{L \in \mathcal{L}(P)} |\des_I(L)|\;.
  \end{align*}
\end{proof}

\begin{remark}
It would be interesting if the previous proof could be refined to a bijection.
\end{remark}

It would be desirable to extend this geometric viewpoint to the previous parts of this paper. 

\begin{remark}
Is there a geometric proof of Theorem~\ref{theorem.homomesy}?  It is natural to interpret a braid hook as the 
codimension 2 face in the order polytope coming from the intersection of the two facets corresponding to the 
relevant edges.  The problem is to again come up with a decomposition of the order polytope by coning 
(now twice!) over all such faces.
\end{remark}

\bibliographystyle{alpha}
\bibliography{paper_reduced_words}

\end{document}